\documentclass[a4paper, reqno, 10pt]{amsart}
\usepackage{amsthm, enumerate, verbatim}
\usepackage{times,latexsym,amssymb, wasysym}
\usepackage[usenames,dvipsnames]{color}
\usepackage[colorlinks,pdfpagelabels,pdfstartview=FitH,bookmarksopen=true,bookmarksnumbered=true,linkcolor=blue,plainpages=false,hypertexnames=false,citecolor=red]{hyperref}

\numberwithin{equation}{section}

\newtheorem{Theorem}{Theorem}[section]
\newtheorem{Proposition}[Theorem]{Proposition}

\newtheorem{Lemma}[Theorem]{Lemma}
\newtheorem{Corollary}[Theorem]{Corollary}

\theoremstyle{definition}
\newtheorem{Definition}{Definition}

\theoremstyle{remark}
\newtheorem{Remark}{Remark}

\def\dx{\,dx}
\def\dy{\,dy}
\def\dt{\,dt}
\def\ds{\,ds}

\DeclareMathOperator{\divergence}{div}

\DeclareMathOperator*{\osc}{osc}

\def\mint_#1{\mathchoice%
          {\mathop{\kern 0.2em\vrule width 0.6em height 0.69678ex depth -0.58065ex
                  \kern -0.8em \intop}\nolimits_{\kern -0.4em#1}}%
          {\mathop{\kern 0.1em\vrule width 0.5em height 0.69678ex depth -0.60387ex
                  \kern -0.6em \intop}\nolimits_{#1}}%
          {\mathop{\kern 0.1em\vrule width 0.5em height 0.69678ex
              depth -0.60387ex
                  \kern -0.6em \intop}\nolimits_{#1}}%
          {\mathop{\kern 0.1em\vrule width 0.5em height 0.69678ex depth -0.60387ex
                  \kern -0.6em \intop}\nolimits_{#1}}}

\newcommand{\aveint}[2]{\mathchoice%
          {\mathop{\kern 0.2em\vrule width 0.6em height 0.69678ex depth -0.58065ex
                  \kern -0.8em \intop}\nolimits_{\kern -0.45em#1}^{#2}}%
          {\mathop{\kern 0.1em\vrule width 0.5em height 0.69678ex depth -0.60387ex
                  \kern -0.6em \intop}\nolimits_{#1}^{#2}}%
          {\mathop{\kern 0.1em\vrule width 0.5em height 0.69678ex depth -0.60387ex
                  \kern -0.6em \intop}\nolimits_{#1}^{#2}}%
          {\mathop{\kern 0.1em\vrule width 0.5em height 0.69678ex depth -0.60387ex
                  \kern -0.6em \intop}\nolimits_{#1}^{#2}}}
                 
\let\originalleft\left
\let\originalright\right
\renewcommand{\left}{\mathopen{}\mathclose\bgroup\originalleft}
\renewcommand{\right}{\aftergroup\egroup\originalright}

\newcommand{\loc}{\textnormal{loc}}
\newcommand{\dist}{\textnormal{dist}}

\newcommand{\R}{{\mathbb R}}

\newcommand{\N}{{\mathbb N}}

\newcommand{\A}{{\mathcal A}}

\newcommand{\eps}{\varepsilon}

\newcommand{\vs}{\vspace{3mm}}

\newcommand{\data}{\texttt{data}}

\begin{document}

\title[The Cauchy-Dirichlet problem for a class of parabolic equations]{The Cauchy-Dirichlet problem\\for a general class of parabolic equations}

\author{Paolo Baroni}
\address{Paolo Baroni\\ Dipartimento di Matematica e Applicazioni ``R. Caccioppoli'', Universit\`a degli Studi di Napoli ``Federico II''\\
I-80125 Napoli, Italy}
\email{paolo.baroni@unina.it}

\author{Casimir Lindfors}
\address{Department of Mathematics and Systems Analysis, Aalto University, P.O. Box 11000, 00076 Aalto, Finland}
\email{casimir.lindfors@aalto.fi}

\allowdisplaybreaks
\date{\today}

\begin{abstract}
We prove regularity results such as interior Lipschitz regularity and boundary continuity for the Cauchy-Dirichlet problem associated to a class of parabolic equations inspired by the evolutionary $p$-Laplacian, but extending it at a wide scale. We employ a regularization technique of viscosity-type that we find interesting in itself.
\end{abstract}

\maketitle
%\tableofcontents

\section{Introduction}
The aim of this paper is the study of the behaviour of solutions to a wide class of nonlinear parabolic equations modeled after
\begin{equation}\label{eq.model}
 u_t-\divergence\Big(\frac{g(|Du|)}{|Du|}Du\Big)=0\qquad\text{in }\qquad\Omega_T:=\Omega\times(0,T)\subset \R^n\times\R,
\end{equation}
$n\geq2$, where $\Omega$ is a bounded domain with $C^{1,\beta}$ boundary and $g:\R_+\to\R_+$ is a $C^1$ function satisfying 
\begin{equation}\label{O.property}
 g_0-1\leq \mathcal O_g(s):=\frac{sg'(s)}{g(s)}\leq g_1-1\qquad \text{for  every $s>0$}
\end{equation}
with $1<g_0\leq g_1<\infty$.  Notice that we can assume $g_0<g_1$ without loss of generality. Indeed, if $\mathcal O_g(s)$ is constant, say $\mathcal O_g(s)=p-1$ for some $p>1$, a simple integration shows that $g(s)=s^{p-1}$ up to a constant factor, and therefore in this case \eqref{eq.model} gives back the evolutionary $p$-Laplacian widely studied in particular by DiBenedetto, see the monograph \cite{DiBenedetto}. This reveals that \eqref{eq.model} is a natural generalization of the $p$-Laplacian, and in effect this class of growth conditions was mathematically introduced exactly in these terms by Lieberman in \cite{Lieb}, even if this kind of condition appears earlier in the applications, see the forthcoming lines. 

We stress that quite a comprehensive study of non-negative solutions to the equation
\begin{equation}\label{kenig}
u_t-\divergence\big[ \varphi'(u)Du\big]=0 
\end{equation}
where the function $\varphi:[0,\infty)\to[0,\infty)$ satisfies
\begin{equation}\label{kenig.dahlberg}
0<a\leq \mathcal O_\varphi(s):=\frac{s\varphi'(s)}{\varphi(s)}\leq\frac1a, \quad\text{for $s>0$}; \qquad 1+a\leq \mathcal O_\varphi(s) \quad\text{for $s>s_0$} 
\end{equation}
for some $a\in(0,1)$ and some $s_0>0$ has been provided by Dahlberg and Kenig \cite{DK,DK2}; see also the books \cite{Dask,Vazquez}. Clearly, while \eqref{kenig} is a generalization of the porous medium equation that happens when $\varphi(u)=u^m$, $m>0$, in the same spirit \eqref{eq.model} can be seen as a generalization of the $p$-Laplacian.

\vs

As in \eqref{kenig.dahlberg}, we shall also consider a more stringent growth assumption for $g$ for large values of its argument. In addition to \eqref{O.property}, we shall assume that there exist constants $c_\ell,\epsilon>0$ such that 
\begin{equation}\label{grow.below.g}
 g(s)\geq c_\ell s^{\frac{n-2}{n+2}+\epsilon}\qquad \text{for any $s\geq 1$}.
\end{equation}
Note that in the $p$-Laplacian case \eqref{grow.below.g} reads precisely as $p>2n/(n+2)$, a completely natural assumption in the theory of the evolutionary $p$-Laplacian operator, see \cite{DiBenedetto, KMper, AM07}. Note moreover that \eqref{grow.below.g} is implied by assuming $g_0>2n/(n+2)$, see Paragraph \ref{prop.g}.

The regularity for the elliptic and variational counterpart of \eqref{eq.model} is quite well understood, see for instance \cite{Lieb, Bg, Cianchi-Mazya, Diening} for the first argument and \cite{DieningBianca, FS, Cianchi, Cianchi-Fusco} for the second, just to cite some indicative references. In the parabolic setting, however, very few results are available, and some of them only in particular cases: to our knowledge, only \cite{Hwang, HwangL, LiebPara1, LiebRussia}, all by Lieberman and Hwang, and the recent \cite{BKa}. 

The difficulty, in particular in finding zero-order results, stems from several facts, the main one perhaps being that the equation has very different behaviour, already in the $p$-Laplacian case, in the degenerate ($p\geq 2$) and singular ($p<2$) cases. In the degenerate case phenomena such as expansion of positivity occur, see \cite{DiBeGianVesp08, Kuu08}, and the diffusion dominates \cite{DiBenedetto.holder}. On the other hand, in the singular case the evolutionary character dominates \cite{Chen-DiB2} and extinction of positive solutions in finite time could happen, see \cite{DiBenedetto}. In our general setting the degenerate case occurs when $s\mapsto g(s)/s$ is increasing, and when it is decreasing we have the singular case. However, it might also happen that $s\mapsto g(s)/s$ has no monotonicity whatsoever, making the handling of the equation all the more difficult. The comprehension of the interaction of these different phenomena is the key for a better understanding of the behaviour of local solutions to \eqref{eq.model}, and in this paper we hope to start to clarify this difficult point, which will be the object of future investigations.

\vspace{3mm}

The class of differential operators we study, besides being quite a general extension of a well-known operator, finds important applications in the applied sciences, also in view of the following observation. Take the convex primitive $G$ of $g$ and consider the general minimization problem
\begin{equation}\label{functional}
u\in u_0+W^{1,1}_0(\Omega)\mapsto\int_\Omega G(|Du|)\dx; 
\end{equation}
it is often convenient to have energies with a precise dependence on $|Du|$ of more general type than monomial (that is, the case of the $p$-Dirichlet energy or appropriate extensions). For instance, in mechanics, fluid dynamics and magnetism, as first approximation it is customary to have dependencies of the energy on the modulus of the gradient of monomial type but with exponent depending on the size of $|Du|$, in order to have mathematical models fitting the experimental data. In this case $g$ is given by the gluing of different monomials (see the example in Paragraph \ref{conc.ex}). At this point, elliptic and parabolic equations having the growth described in \eqref{eq.model} arise naturally as Euler equations or flows of the functional in \eqref{functional}. In \cite{Shiffman}, for instance, the two-dimensional stationary, irrotational subsonic flow of a compressible fluid is described using an energy defined in the following way:
\begin{equation}\label{a.shiffman}
G(s)=-\Big(1-\frac{\gamma-1}2s^2\Big)^{\frac\gamma{\gamma-1}} \quad\text{for small $s$,}\qquad G(s)=\text{quadratic \quad otherwise}, 
\end{equation}
where $\gamma\in(1,2)$ is the exponent in the law $p\approx \rho^\gamma$ characteristic of polytropic gases.

More in general, see \cite{Bers, Dong, Acta}, one is lead to consider quasilinear static equations in dimension two and three of the type
\[
\divergence\big[\rho(|Du|^2)Du\big]=0,
\]
with $Du$ representing the {\em velocity field of the flow} and $q=|Du|$ being the {\em speed of the flow}. In this context one introduces the {\em Mach number}
\[
M^2\equiv [M(q)]^2:=-\frac{2 q^2}{\rho(q^2)}\rho'(q^2)
\]
(note that we must have $\rho'<0$). In our context, where $g(s)=\rho(s^2)s$, we compute $\mathcal O_g(s)=1-M(s)^2$. The general theory asserts that a point is {\em elliptic}
if $M<1$ and in this case the flow is {\em subsonic}, while if $M>1$ the point is {\em hyperbolic} and the flow there is {\em supersonic}. If $M=1$ the flow is called {\em sonic}.
A solution of the boundary value problem is called a subsonic (supersonic) flow according to whether all points are subsonic (supersonic); note that {\em mixed, or transonic flows} can exists, with obvious meaning. However, if for some reason we know that the flow maintains a controlled, small speed $q$, then the problem falls in the class of operators we consider; the approximation in \eqref{a.shiffman} is a way to study flows in the subsonic regime.

\vspace{3mm}

The object of our study will be the Cauchy-Dirichlet problem
\begin{equation}\label{general.equation}
\begin{cases}
u_t-\divergence \A(Du) =0\qquad&\text{in }\Omega_T, \\[3mm]
u=\psi &\text{on $\partial_p\Omega_T$,}
\end{cases}
\end{equation}
where $\A:\R^n\to\R^n$ is a $C^1$ vector field modeled after the one appearing in \eqref{eq.model}. In particular, we assume it satisfies the following ellipticity and growth conditions:
\begin{equation}\label{assumptionsO}
\begin{cases}
\displaystyle{\langle D\A(\xi)\lambda, \lambda\rangle \geq \nu \frac{g(|\xi|)}{|\xi|}|\lambda|^2}\\[3mm]
\displaystyle{|D\A(\xi)| \leq L\,\frac{g(|\xi|)}{|\xi|}}\\[2mm]
\end{cases},
\end{equation}
for any $\xi\in\R^n\setminus\{0\}, \lambda\in\R^n$ and with structural constants $0<\nu\leq1\leq L$; we assume without loss of generality that $\A(0)=0$. The function $g$ is a $C^1$ function as in \eqref{eq.model}, satisfying only \eqref{O.property} and \eqref{grow.below.g}. For what concerns $\psi$, we assume it to be continuous in $\partial_p\Omega_T$ with modulus of continuity $\omega_\psi$ with respect to the natural distance $\dist_{{\rm par},G}$, that is, there exists a continuous, concave function $\omega_\psi:\R_+\to\R_+$ with $\omega_\psi(0)=0$ such that
\[
|\psi(x,t)-\psi(y,s)|\leq \omega_\psi\big(\max\{|x-y|,[G^{-1}(1/|t-s|)]^{-1}\}\big)
\]
for every $(x,t),(y,s)\in\partial_p\Omega_T$. As already mentioned, $\Omega$ is a bounded domain of $\R^n$, $n\geq2$, whose boundary is of class $C^{1,\beta}$ for some $\beta\in(0,1)$; we shall provide some more details at the beginning of Section \ref{second}.

\vs

In this setting, we state the main result of our paper, which concerns at the same time the existence and regularity of a (unique) solution to \eqref{general.equation}.

 \begin{Theorem}\label{Ex}
There exists a unique solution $u$, in the sense of Definition \ref{solution.CD}, to the Cauchy-Dirichlet problem \eqref{general.equation}, where the vector field $\A$ satisfies the assumptions \eqref{assumptionsO}, with $g\in C^1(\R_+)$ satisfying \eqref{O.property} and \eqref{grow.below.g}. In particular, $u$ is continuous up to the boundary and moreover if the boundary datum $\psi$ is H\"older continuous with respect to the natural metric $\dist_{{\rm par},G}$ defined in \eqref{dist.par.G}, then so is $u$.
 \end{Theorem}

 The following theorem gives some properties together with quantitative estimates for the solution described in the previous statement.
 
 \begin{Theorem}\label{Lip}
Let $u$ be the solution to \eqref{general.equation} given by Theorem~\ref{Ex}. Then $u$ is locally Lipschitz continuous and the following estimate holds:
\begin{equation}\label{loc.lipschitz}
\|Du\|_{L^{\infty}(Q_R)}\leq c\biggl(\mint_{Q_{2R}}\big[G(|Du|)+1\big]\dx\dt\biggr)^{\max\left\{\frac12,\frac2{\epsilon(n+2)}\right\}}
\end{equation}
for every parabolic cylinder $Q_{2R}\Subset\Omega_T$. The constant $c$ depends on $n,g_0,g_1,\nu,L,\epsilon$ and $c_\ell$. Moreover, there exists a modulus of continuity $\omega_u:\R_+\mapsto\R_+$ depending on $n,g_0,g_1,\nu,L,\epsilon,c_\ell,\|\psi\|_{L^\infty},\omega_\psi,\partial\Omega$ such that
\begin{equation}\label{exp.continuity}
|u(x,t)-u(y,s)|\leq \omega_u\big(\max\big\{|x-y|,\big[G^{-1}\big(1/|t-s|\big)\big]^{-1}\big\}\big) 
\end{equation}
for every $(x,t),(y,s)\in {\overline{\Omega_T}}^p$. 
 \end{Theorem}
We refer the reader to Paragraph \ref{notation} for the definitions of the standard parabolic cylinders $Q_R(x_0,t_0)$ and of the parabolic closure of $\Omega_T$. We also mention that in the standard case of the evolutionary $p$-Laplacian our estimate \eqref{loc.lipschitz} gives  back exactly the gradient $\sup$-estimate available for degenerate and singular equations, see \cite[Chapter VIII, Theorems 5.1 \& 5.2]{DiBenedetto}.

\begin{Remark}\label{more.general.conditions}
Theorems \ref{Ex} and \ref{Lip} hold for a wider class of operators generalizing \eqref{eq.model}, which allow the presence of a function $g$ that is not $C^1$ but merely Lipschitz. Indeed, we may consider Lipschitz functions $g:\R_+\to\R_+$ satisfying \eqref{O.property} almost everywhere and vector fields $\A:\R^n\to\R^n$ in $W^{1,\infty}(\R^n)$ satisfying the monotonicity and Lipschitz assumptions
\begin{equation}\label{assumptionsO.weak}
\begin{cases}
\displaystyle{\langle \A(\xi_1)-\A(\xi_2),\xi_1-\xi_2\rangle \geq \nu \frac{g(|\xi_1|+|\xi_2|)}{|\xi_1|+|\xi_2|}|\xi_1-\xi_2|^2}\\[3mm]
\displaystyle{|\A(\xi_1)-\A(\xi_2)| \leq L\, \frac{g(|\xi_1|+|\xi_2|)}{|\xi_1|+|\xi_2|}|\xi_1-\xi_2|},
\end{cases},
\end{equation}
for every $\xi_1,\xi_2\in\R^n$ such that $|\xi_1|+|\xi_2|\neq 0$ and for some $0<\nu\leq 1\leq L$. For a proof of this fact see the end of Section \ref{conclusion}.
\end{Remark}

\subsection{Novelties and technical tools}
We believe that the main interest of this paper, apart from the results of Theorems \ref{Ex} and \ref{Lip} themselves (that will be used for instance in \cite{CL}), is the development of some tools for the treatment of the difficult equation \eqref{general.equation} (see Paragraph \ref{conc.ex}). We prove the Lipschitz estimate as an {\em a priori estimate} for problems enjoying further regularity. Instead of using a regularization of the type used in \cite{Lieb, Hwang, LiebPara1}, the regularization we employ is of viscosity type, closer to that in \cite{AMS}: we consider a vector field of the type 
\[
\A_\eps(\xi):=(\phi_\eps\ast\A)(\xi)+\eps \big(1+|\xi|\big)^{p-2}\xi,\qquad\xi\in\R^n,\eps\in(0,1),
\]
where $p\gg1$ is a large exponent and $\{\phi_\eps\}$ a family of mollifiers. This allows us to overcome the difficulties of deriving regularity estimates for the approximant problems, which we were not able to find in the literature. At this point continuity up to the boundary becomes an essential ingredient in the proof of the convergence, as well as the fact that we are solving a Cauchy-Dirichlet problem and therefore have a uniform bound on $\|u_\eps\|_{L^\infty}$ given by the maximum principle. 

\vs

We use the a priori Lipschitz continuity (and the further regularity) of the approximating solutions in a way inspired by \cite{KMibe}. First, we employ the fact that the function $v=|Du|^2$ is a subsolution to a similar problem, see Lemma \ref{vsub}. Then, we define an appropriate intrinsic geometry (see \eqref{cyl.int}) depending on the growth of the approximating vector field $\A_\eps$, which allows us to rebalance estimates, in the sense that the weight appearing in the Caccioppoli estimate for the equation satisfied by $v$ turns out to be essentially constant, see \eqref{rehomog}. Here the fact that we can bound the supremum of $Du$, and thus of $v$, from above is essential. Finally, we conclude the proof using an argument based on an alternative in order to get rid of the possible dependence on $\eps$ in terms of the aforementioned geometry, depending in turn on the growth of $\A_\eps$.

\section{Preliminary material: notation, the function $g$, miscellanea}\label{second}

For what concerns $\partial\Omega$, we assume that there exists a radius $R_\Omega>0$ such that for every point $x_0\in\partial\Omega$ there is a unit vector $\hat e_{x_0}$ such that the restriction of $\partial \Omega$ is a graph of a $C^{1,\beta}$ function in $B_{R_\Omega}$ along the $\hat e_{x_0}$ direction, in the following sense: with $T$ being an orthogonal transformation that maps $\hat e_{x_0}$ into $(0,0,\dots,0,1)$, for every $0<r\leq R_\Omega$ it holds 
\[
T^{-1}(\partial \Omega-x_0)\cap \big(B_r'\times(-r,r)\big)={\rm graph}\,\theta
\]
(see below for the precise meaning of these symbols) with $\theta\equiv \theta_{x_0}\in C^{1,\beta}(B_r')$, $\theta(B_r')\subset(-r,r)$ and the $C^{1,\beta}$ norm of $\theta$ uniformly bounded:
\begin{equation*}%\label{boundary.beta}
[\theta]_{C^{1,\beta}}\leq \Theta.
\end{equation*}
Note that without loss of generality, we can take $\hat e_{x_0}$ as the inner normal vector in $x_0$: $\{v:\langle v,\hat e_{x_0}\rangle=0\}$ is the tangent hyperplane to $\Omega$ in $x_0$; therefore $D\theta(0)=0$. $D\theta$ is the full gradient of $\theta$ with respect to its $n-1$ variables. Finally, by saying that a constant depends on $\partial\Omega$, we shall mean it depends on $\Theta$.

\subsection{Notation}\label{notation}
We denote by $c$ a general constant {\em always larger than or equal to one}, possibly varying from line to line; relevant dependencies on parameters will be emphasized using parentheses, i.e., ~$c_{1}\equiv c_1(n,p,q)$ means that $c_1$  depends on $n,p,q$. For the ease of notation, we shall also use the following abbreviation:
\[
\data:=\{n,g_0,g_1,\nu,L\}.
\]
We denote by $$ B_R(x_0):=\{x \in \R^n \, : \,  |x-x_0|< R\}$$ the open ball with center $x_0$ and radius $R>0$; when clear from the context or otherwise not important, we shall omit denoting the center as follows: $B_R \equiv B_R(x_0)$. The standard parabolic cylinder is defined as
\[
Q_R(x_0,t_0):=B_R(x_0)\times(t_0-R^2,t_0),
\]
while we define the natural cylinder as
\[
Q_R^G(x_0,t_0):=B_R(x_0)\times(t_0-[G(1/R)]^{-1},t_0).
\]
The latter is strictly linked to the scaling of the equation, see Paragraph \ref{geometry}. Unless otherwise explicitly stated, different balls and cylinders in the same context will have the same center. We shall denote, for a  factor $\alpha>0$, by $\alpha B_R$ the  ball $B_{\alpha R}$ and by $\alpha Q_R(x_0,t_0)$ the cylinder $B_{\alpha R}(x_0)\times(t_0-(\alpha R)^2,t_0)$; similarly for $\alpha Q_R^G(x_0,t_0)$. The parabolic boundary of a cylindrical domain $\mathcal K=\mathcal D\times \Gamma$, where $\mathcal D$ is an open domain and $\Gamma$ an open interval of the real line, is defined as
\[
\partial_p\mathcal K:=\big(\overline{\mathcal D}\times\inf \Gamma\big)\cup\big(\partial \mathcal D\times \Gamma\big).
\]
Naturally, the parabolic closure of $\mathcal K$ is then ${\overline{\mathcal K}}^p:=\mathcal K\cup\partial_p\mathcal K$. Accordingly with the customary use in the parabolic setting, when considering a sub-cylinder $\mathcal K$ (as above) compactly contained in $\Omega_T$, we shall mean that $\mathcal D\Subset\Omega$ and  $0<\inf \Gamma<\sup \Gamma\leq T$; we will write in this case $\mathcal K\Subset\Omega_T$. By $\partial \Omega-x_0$ we mean the set $\{x\in\R^n:x+x_0\in\partial\Omega\}$. The standard parabolic distance is
\[
\dist_{\rm par}\big((x,t),(y,s)\big):=\max\big\{|x-y|,\sqrt{|t-s|}\big\}
\]
for any $(x,t),(y,s)\in\R^{n+1}$, while a distance strictly related to the scaling properties of the differential operator is \begin{equation}\label{dist.par.G}
\dist_{{\rm par},G}\big((x,t),(y,s)\big):=\max\Big\{|x-y|,\Big[G^{-1}\Big(\frac1{|t-s|}\Big)\Big]^{-1}\Big\}.
\end{equation}
Note that $Q_R^G(x_0,t_0)=\{(x,t)\in\R^{n+1}\,:\,\dist_{{\rm par},G}((x,t),(x_0,t_0))<R, t<t_0\}$ and similarly for $Q_R(x_0,t_0)$. Accordingly we define the parabolic distance between sets as
\[
\dist_{\rm par}(A,B):=\inf_{\substack{(x,t)\in A\\(y,s)\in B}}\dist_{\rm par}\big((x,t),(y,s)\big)
\]
for $A,B\subset\R^{n+1}$; similarly for $\dist_{{\rm par},G}(A,B)$. 

At a certain point it will be useful to split $\R^n=\R^{n-1}\times\R$. We agree here that we shall write a point $x\in \R^n$ as $(x',x_n)\in \R^{n-1}\times\R$; moreover, with $B_r'(x_0')$ we shall denote the ball of $\R^{n-1}$ with radius $r$ and center $x_0'\in \R^{n-1}$.

With $\mathcal B \subset \R^\ell$ being a measurable set, $\chi_{\mathcal B}$ denotes its characteristic function. If furthermore $\mathcal B$ has positive and finite measure and $f: \mathcal B \to \R^{k}$ is a measurable map, we shall denote by  
\[
   (f)_{\mathcal B} \equiv \mint_{\mathcal B}  f(y) \,dy  := \frac{1}{|\mathcal B|}\int_{\mathcal B}  f(y) \,dy
\]
the integral average of $f$ over $\mathcal B$. If $\mathcal B$ is a cylinder, $\mathcal B:=K\times\Gamma\subset\R^{n+1}$, then we shall denote the slicewise average by
\[
 (f)_{K}(\tau) := \mint_{K}  f(y,\tau) \dy.
\]
for almost every $\tau\in \Gamma$. By $\sup$ we shall mean possibly the essential supremum, and similarly for $\inf$. We shall also as usual denote 
\[
\osc_{\mathcal B}f:=\sup_{\mathcal B}f-\inf_{\mathcal B}f,\qquad [f]_{C^{0,\gamma}(\mathcal B)} := \sup_{\overset{x,y \in \mathcal B}{x \neq y}} \, \frac{|f(x)-f(y)|}{|x-y|^\gamma}.
\]
$D_if:=\partial f/\partial x_i$, for $i\in\{1,\dots,n\}$, will stand for the partial derivative of $f$ in the $\hat e_i$ direction, and $D^2_{i,j}f$ will denote $\partial^2 f/\partial x_i\partial x_j$. Here $\hat e_i$ is the $i$-th element of the standard orthonormal basis of $\R^n$.
By $2^*$ we shall denote the Sobolev conjugate exponent of $2$, with the agreement that in the case $n=2$ we fix the value of $2^*$ as $4$, i.e.,
\begin{equation}\label{convention}
  2^*:=
\begin{cases}
   \displaystyle{\frac{2n}{n-2}} & n>2,\\[3mm] 
      4 & n=2.
\end{cases} 
\end{equation}
 With $s$ being a real number, we shall denote $s_+:=\max\{s, 0\}$ and $s_-:=\max\{-s, 0\}$. For a vector $\xi=(\xi_1,\dots,\xi_n)\in\R^n$, ${\rm diag}\,\xi$ denotes the diagonal matrix $(\xi_i\delta_{i,j})_{i,j=1}^n$. Finally, $\R_+:=[0,\infty)$, $\N$ is the set $\{1,2,\dots\}$ and $\N_0=\N\cup\{0\}$.

By ``equation structurally similar to \eqref{general.equation}$_1$'' we mean an equation of the type $\partial_tu-\divergence \widetilde{\A}(Du)=0$ with $\widetilde{\A}$ satisfying assumptions \eqref{assumptionsO} with $\nu,L$ and $g$ replaced by $\tilde\nu,\widetilde L$ and $\tilde g$. Both $\tilde\nu,\widetilde L$ will depend on $\data$, while $\tilde g$ will satisfy \eqref{O.property} and \eqref{grow.below.g} with $\widetilde g_0$, $\widetilde g_0$, $\widetilde{c_\ell}$ depending on $\data$ and $c_\ell$. 
% Similarly for a vector field $\widetilde{\A}$ satisfying ``structural assumptions similar to \eqref{assumptionsO}''. Note however that it must still hold the normalization condition $\tilde g(1)=g(1)$ - and it will. 

\subsection{Properties of $\boldsymbol{g}$}\label{prop.g}
Without loss of generality we  assume that
\begin{equation}\label{normalization}
 \int_0^1 g(\rho)\,d\rho = 1.
\end{equation}
Since \eqref{O.property} implies that the map $r \mapsto g(r)r^{-(g_0-1)}$ is increasing, while $r \mapsto g(r)r^{-(g_1-1)}$ turns out to be decreasing, we have
\[
 \min\left\{\alpha^{g_0-1},\alpha^{g_1-1}\right\}g(r)\leq g(\alpha r)\leq \max\left\{\alpha^{g_0-1},\alpha^{g_1-1}\right\}g(r)
\]
for every $r,\alpha>0$; clearly $g(0)=0$ and $\lim_{r\to\infty}g(r)=\infty$. Since moreover $g$ is strictly increasing, it has a strictly increasing inverse function $g^{-1}\in C^1(\R_+)$ with
\[
 \left(g^{-1}\right)'(r) = \frac{1}{g'(g^{-1}(r))}\qquad\text{for every $r>0$. }
\]
Using \eqref{O.property} we then see that also $g^{-1}$ satisfies an Orlicz-type condition 
\begin{equation}\label{O.property2}
 \frac{1}{g_1-1} \leq \frac{(rg^{-1})'(r)}{g^{-1}(r)} \leq \frac{1}{g_0-1}\qquad\text{for every $r>0$. }
\end{equation}
Therefore, anything derived from \eqref{O.property} for $g$ holds for $g^{-1}$ with $g_0-1$ and $g_1-1$ replaced by 
${1}/(g_1-1)$ and $1/(g_0-1)$, respectively.  

Define the function $G: \R_+\to\R_+$ as 
\begin{equation}\label{G}
 G(r) := \int_0^r g(\rho)\,d\rho.
\end{equation}
Clearly $G'(r) = g(r) > 0$ and $G''(r) = g'(r) > 0$ implying that $G$ is both strictly increasing and strictly convex in $(0,\infty)$. Moreover, $G(0) = 0$ and $G(1) = 1$ due to \eqref{normalization}. We also define $1/G(1/s)=1/G^{-1}(1/s)=0$ for $s=0$. It is simple to check by integrating the function $r\mapsto rg(r)$ by parts and using \eqref{O.property} that also
\begin{equation}\label{O.property3}
 g_0 \leq \frac{G'(r)r}{G(r)} \leq g_1
\end{equation}
holds true for $r>0$.  

Define the Young complement of $G$ as 
%\[
% .
%\]
%Since 
%\[
% \int_0^r g^{-1}(\rho)\,d\rho = \int_0^{g^{-1}(r)}\rho g'(\rho)\,d\rho = rg^{-1}(r) - \int_0^{g^{-1}(r)}g(\rho)\,d\rho,
%\]
% Equality holds precisely when $s=g^{-1}(r)$. Thus, 
%\[
% \widetilde G(r) = rg^{-1}(r) - G(g^{-1}(r)) \leq \sup_{s\geq 0}\,(rs - G(s)) \leq \widetilde G(r),
%\]
%from which we see that $\widetilde G$ could also have been defined as 
\begin{equation}\label{conjugateG}
\widetilde G(r) = \sup_{s> 0}\,(rs - G(s))\qquad \text{or}\qquad \widetilde G(r) := \int_0^r g^{-1}(\rho)\,d\rho;
\end{equation}
in our setting these definitions are equivalent, see \cite{Rao}. Note that the Young's inequality 
\begin{equation}\label{Young.easy}
 sr\leq G(s) + \widetilde G(r) 
\end{equation}
holds true for every $r,s> 0$ and by \eqref{O.property2} and the second definition in \eqref{conjugateG} also $\widetilde G$ satisfies an Orlicz-type condition 
\begin{equation}\label{O.property4}
 \frac{g_1}{g_1-1} \leq \frac{\widetilde G'(r)r}{\widetilde G(r)} \leq \frac{g_0}{g_0-1}.
\end{equation}
Now starting from \eqref{O.property3} and \eqref{O.property4}, we deduce precisely as for $g$ the inequalities 
\begin{equation}\label{GDelta2}
 \min\left\{\alpha^{g_0},\alpha^{g_1}\right\}G(r)\leq G(\alpha r)\leq \max\left\{\alpha^{g_0},\alpha^{g_1}\right\}G(r),
\end{equation}
and 
\[
 \min\left\{\alpha^{\frac{g_1}{g_1-1}},\alpha^{\frac{g_0}{g_0-1}}\right\}\widetilde G(r)\leq \widetilde G(\alpha r)
 \leq \max\left\{\alpha^{\frac{g_1}{g_1-1}},\alpha^{\frac{g_0}{g_0-1}}\right\}\widetilde G(r)
\]
for every $\alpha,r\geq 0$. These, together with Young's inequality \eqref{Young.easy}, imply for $0<\varepsilon<1$ 
\[
 sr  \leq G(\varepsilon^{\frac{1}{g_0}}s) 
 + \widetilde G(\varepsilon^{-\frac{1}{g_0}}r) \leq \varepsilon G(s) + c(g_0,\varepsilon)\widetilde G(r).
\]
Another useful property is
\begin{equation*}%\label{Yf}
 \widetilde G\Big(\frac{G(r)}{r}\Big)  \leq G(r)\qquad\text{for every $r>0$,}
\end{equation*}
see again \cite{Rao} for the easy proof.

From the second assumption of \eqref{assumptionsO} we easily derive an upper bound for $\A$. Indeed, when $\xi\in\R^n\setminus\{0\}$ we 
have 
\begin{multline}\label{A.bound}
 |\A(\xi)|\leq |\xi|\int_0^1|D\A(s\xi)|\,ds\leq L|\xi|\int_0^1\frac{g(s|\xi|)}{s|\xi|}\,ds\\\leq c({L},{g_0})\int_0^{|\xi|}g'(r)\,dr \leq c(g_0,g_1,L)\frac{G(|\xi|)}{|\xi|};
\end{multline}
this holds also for $\xi=0$ by our conventions, since $\A(0)=0$. Similarly, the first assumption of \eqref{assumptionsO} yields 
\begin{equation}\label{A.ellipticity}
\left\langle\A(\xi),\xi\right\rangle = \int_0^1\left\langle D\A(s\xi)\xi,\xi\right\rangle\ds\geq c(g_1,\nu)|\xi|\int_0^{|\xi|}g'(r)\,dr\geq c(g_0,g_1,\nu)G(|\xi|).
\end{equation}
We define the quantity $V_g:\R^n\to\R^n$ by 
\[
 V_g(\xi) = \left(\frac{g(|\xi|)}{|\xi|}\right)^{\frac{1}{2}}\xi
\]
when $\xi\neq 0$ and set $V_g(0)=0$. Clearly $V_g$ is a continuous bijection of $\R^n$ and, moreover, has a continuous inverse by the inverse function theorem. Furthermore, the following monotonicity formula holds true: 
% Indeed, 
% \[
%  \det(DV_g(\xi))=\frac{1}{2}\left(\frac{g(|\xi|)}{|\xi|}\right)^{\frac{n}{2}} \left(1+\frac{|\xi|g'(|\xi|)}{g(|\xi|)}\right)>0
% \]
% whenever $\xi\neq 0$ and $V_g^{-1}$ can be continuously extended to be zero at the origin. such that
\begin{equation}\label{monotonicity}
  \langle \A(\xi_1)-\A(\xi_2),\xi_1-\xi_2 \rangle\geq c\frac{g(|\xi_1|+|\xi_2|)}{|\xi_1|+|\xi_2|}|\xi_1-\xi_2|^2\geq c\,|V_g(\xi_1)-V_g(\xi_2)|^2 
\end{equation}
for a constant $c\equiv c(g_0,g_1,\nu)$ and for every $\xi_1,\xi_2\in\R^n$, see \cite{Diening, DieningBianca}.

% \begin{Lemma}\label{monotonicity}
%  \emph{(Strict monotonicity)} There exists a constant $c\equiv c(g_0,g_1,\nu)$ such that 
%  \[
%   \langle \A(z_1)-\A(z_2),z_1-z_2 \rangle\geq cg'(|z_1|+|z_2|)|z_1-z_2|^2
%  \]
%  for every $z_1,z_2\in\R^n$. 
% \end{Lemma}
% \begin{proof}
% The claim is clearly true when $z_1=z_2=0$. Let then $z_1,z_2\in\R^n$ be such that either $z_1\neq 0$ or $z_2\neq 0$. A standard result 
% \[
%  \int_0^1 |sz_1 + (1-s)z_2|\,ds \geq \frac{1}{4}(|z_1| + |z_2|),
% \]
% together with \eqref{O.property}, \eqref{Orlicz3}, \eqref{GDelta2}, and Jensen's inequality ($G$ is convex), 
% yields 
% \[
%  \begin{split}
%  \int_0^1\frac{g(|sz_1+(1-s)z_2|)}{|sz_1+(1-s)z_2|}\,ds 
%   &\geq g_0\int_0^1\frac{G(|sz_1+(1-s)z_2|)}{(|sz_1+(1-s)z_2|)^2}\,ds\\
%   &\geq \frac{g_0}{(|z_1|+|z_2|)^2}G\left(\int_0^1|sz_1+(1-s)z_2|\,ds\right)\\
%   &\geq \frac{g_0}{4^{g_1}}\frac{G(|z_1|+|z_2|)}{(|z_1|+|z_2|)^2}\\
%   &\geq \frac{g_0}{4^{g_1}g_1(g_1-1)}g'(|z_1|+|z_2|).
%  \end{split}
% \]
% Thus, by using the first inequality of assumption \eqref{assumptions}, we obtain 
% \[
%  \begin{split}
%   \langle \A(z_1)-\A(z_2),z_1-z_2 \rangle &= \left\langle \int_0^1 D\A(sz_1+(1-s)z_2)(z_1-z_2)\,ds,z_1-z_2 \right\rangle\\
%   &\geq \nu\int_0^1\frac{g(|sz_1+(1-s)z_2|)}{|sz_1+(1-s)z_2|}\,ds|z_1-z_2|^2\\
%   &\geq \frac{\nu g_0}{4^{g_1}g_1(g_1-1)}g'(|z_1|+|z_2|)|z_1-z_2|^2,
%  \end{split}
% \]
% as required.
% \end{proof}

\subsection{A concrete example}\label{conc.ex}

We give here a nontrivial example of a Lipschitz function $g$ satisfying our assumptions - see Remark \ref{more.general.conditions}. This example is inspired by \cite{Lieb}. In particular we want to demonstrate the possibility that $g$ oscillates between degenerate and singular behaviour. Suppose $2n/(n+2)<g_0<g_1$ and set $\delta=(g_1-g_0)/3>0$. Define the sequence $s_k=2^{2^k}$ for $k\in\N_0$ and the function 
\[
 g(s)=
 \begin{cases}
   s^{g_0-1+\delta}, & 0<s<2\\[1mm]
   s_{2k+1}^{-\delta}s^{g_1-1}, & s_{2k}\leq s < s_{2k+1}\\[1mm]
   s_{2k+2}^{\delta}s^{g_0-1}, & s_{2k+1}\leq s < s_{2k+2}
\end{cases}.
\]
Clearly $g$ is Lipschitz and it satisfies \eqref{O.property}. Moreover, \eqref{normalization} holds after scaling by a suitable normalization constant. We observe that 
\begin{align*}
 \limsup_{s\to\infty}\frac{g(s)}{s}&=
 \begin{cases}
   \infty, & g_1>2+\delta\qquad\text{(iff \quad $g_0+2g_1>6$)}\\
   1, & g_1=2+\delta\qquad\text{(iff \quad $g_0+2g_1=6$)}\\
   0, & g_1<2+\delta\qquad\text{(iff \quad $g_0+2g_1<6$)}
\end{cases},\\\\
\liminf_{s\to\infty}\frac{g(s)}{s}&=
 \begin{cases}
   \infty, & g_0>2-\delta\qquad\text{(iff \quad $2g_0+g_1>6$)}\\
   1, & g_0=2-\delta\qquad\text{(iff \quad $2g_0+g_1=6$)}\\
   0, & g_0<2-\delta\qquad\text{(iff \quad $2g_0+g_1<6$)}\\
\end{cases}.
\end{align*}
By taking $g_0=2-\frac{3}{2n}, g_1=2+\frac{3}{2n}$ we obtain a particularly interesting case, that is, we have $\liminf_{s\to\infty}g(s)/s=0$ but $\limsup_{s\to\infty}g(s)/s=\infty$. Furthermore, if we consider the function 
\[
 \widetilde g(s)=\frac{1}{g(1/s)},
\]
we find similar behaviour as $s\to 0$. This is to say, we can build a structure function $g$ (and accordingly a vector field $\A$ as in \eqref{eq.model})  that, for $\ell\in\N$, along the sequence $\{\ell^{-k}\}_{k\in\N_0}$ the function $g(s)/s$ is at the same time as large and as close to zero as we wish, and therefore it does not enjoy any monotonicity properties. This gives a clue about the difficulty of the application of De Giorgi-type methods, in particular when they have to be matched with intrinsic geometries: note that the expressions of the type $G(s)/s^2\approx g(s)/s$ appear already in the energy estimate for \eqref{eq.model}, see Lemma \ref{caccioppoli}. On the other hand, when the quantity $g(|Du|)/|Du|$ is known to be under control, then the equation becomes treatable, see for instance Proposition \ref{reverseholder2} and in particular \eqref{rehomog}.

\subsection{Orlicz spaces} For $G$ as in \eqref{G}, a measurable function $u:A\to\R$, $A\subset\R^k$, $k\in\N$ belongs to the Orlicz space $L^G(A)$ if it satisfies 
\[
 \int_{ A}G(|u|)\dx < \infty.
\]
The space $L^G(A)$ is a vector space, since $G$ satisfies the $\Delta_2$-condition \eqref{GDelta2}, and it can be shown to be a 
Banach space if endowed with the Luxemburg norm 
\[
\|u\|_{L^G(A)} := \inf\left\{\lambda>0:\int_AG\Big(\frac{|u|}{\lambda}\Big)\dx\leq 1\right\}.
\]
A function $u$ belongs to $L_{\textrm{loc}}^G(A)$, if $u\in L^G(A')$ for every $A'\Subset A$. If also the weak gradient 
of $u$ belongs to $L^G(A)$, we say that $u\in W^{1,G}(A)$. The corresponding space with zero boundary values, denoted $W^{1,G}_0(A)$, is the completion of $C^{\infty}_c(A)$ under the norm 
\[
\|u\|_{W^{1,G}(A)}:=\|u\|_{L^G(A)}+\|Du\|_{L^G(A)}. 
\]
We denote by $V^G(\Omega_T)$ the space of functions $u\in L^G(\Omega_T)\cap L^1(0,T;W^{1,1}(\Omega))$ for which also the weak spatial gradient $Du$ belongs to $L^G(\Omega_T)$. The space $V^G(\Omega_T)$ is also a Banach space with the norm 
\[
\|u\|_{V^G(\Omega_T)} := \|u\|_{L^G(\Omega_T)}+\|Du\|_{L^G(\Omega_T)}.
\]
Moreover, we denote by $V_0^G(\Omega_T)$ the space of functions $u\in V^G(\Omega_T)$ that belong to $W^{1,G}_0(\Omega)$ for almost every $t\in(0,T)$, while the localized version $V^G_\loc(\Omega_T)$ is defined, as above, in the customary way. We also shorten
\[
V^{2,G}(\Omega_T):=L^\infty\left(0,T;L^2(\Omega)\right)\cap V^G(\Omega_T)
\]
and similarly for the localized and the zero trace versions. We shall moreover denote $V^{2,p}(\Omega_T)$, for $p>1$, the space $V^{2,G}(\Omega_T)$ for the choice $G(s)=s^p$.
\vs

\subsection{The concept of solution and consequences}
We fix here the notions of solution employed in this paper.
\begin{Definition}\label{weak.solution}
 A function $u$ is a \emph{weak solution} to \eqref{general.equation}$_1$ in a cylindrical domain $\mathcal K\subset\R^{n+1}$, with the vector field $\A$ satisfying the assumptions \eqref{assumptionsO}, if $u\in V^{2,G}_{\rm loc}(\mathcal K)$ and it satisfies the weak formulation
\begin{equation}\label{weaksolution}
\int_{\mathcal K}\big[-u\partial_t\eta+\langle\A(Du), D\eta\rangle\big]\dx\dt = 0
\end{equation}
for every test function $\eta \in C_c^\infty(\mathcal K)$. If instead of equality we have the $\leq(\geq)$ sign for every nonnegative 
$\eta\in C_c^\infty(\mathcal K)$, we say that $u$ is a \emph{weak subsolution (supersolution)} in $\mathcal K$. 
\end{Definition}
%\begin{Remark}
%If $u$ is a weak supersolution of \eqref{general.equation}, it is easy to see that $-u$ is a weak subsolution of the same equation with 
%$\A(Du)$ replaced by $-\A(-Du)$. However, since the latter vector field also satisfies the structural conditions \eqref{assumptionsO}, 
%we shall say, with a little abuse of notation, that $-u$ is a subsolution to \eqref{general.equation} itself.....
%\end{Remark}
\begin{Definition}\label{solution.CD}
 A function $u$ is a solution to the Cauchy-Dirichlet problem \eqref{general.equation} if $u\in C^0(\overline\Omega_T)$ is a weak solution to \eqref{general.equation}$_1$ in $\Omega_T$ and moreover $u=\psi$ pointwise on $\partial_p\Omega_T$. 
\end{Definition}

\vs

A very useful formulation, equivalent to \eqref{weaksolution}, is the one involving Steklov averages. Indeed, the mild regularity of a solution does not allow us to use it as a test function. Furthermore, it is sometimes useful to have a weak formulation allowing for test functions independent of time, or test functions possibly vanishing only on the parabolic boundary of a cylinder. Apart from mollification, the possible way to have such properties involve the so-called Steklov averaging regularization of a function: for $f:\mathcal K=\mathcal D\times(t_1,t_2)\to\R$ measurable and $0<|h|\ll1$ appropriate, it is defined as 
 \[
f_h(x,t) := \frac{1}{h}\int_{t-h}^tf(x,s)\,ds\qquad\text{for $(x,t) \in \mathcal D\times (t_1+h,t_2)$};
 \]
note that we employ the {\em backward regularization}. If $f\in L^q(\mathcal K)$ for some $q\geq1$, then $f_h\to f$ in $L^q(\mathcal D\times (t_1+\eps,t_2))$ for every $\eps>0$; the same holds in the $L^G$ spaces. Moreover, if $f\in C^0(t_1,t_2; L^q(\mathcal D))$ then $f_h(\cdot,\tau)\to f(\cdot,\tau)$ in $L^q(\mathcal D)$ for a.e. $\tau \in (t_1+\eps,t_2)$ and for every $\eps>0$.

At this point it is quite easy to infer the following slicewise formulation for weak solutions (see \cite{DiBenedetto}) using density arguments with respect to the spatial variable:
\begin{equation}\label{Steklov.formulation}
\int_{\mathcal D}\big[\partial_t u_h(\cdot,\tau)\eta +\big\langle[\A(Du)]_h(\cdot,\tau), D\eta\big\rangle\big]\dx= 0
\end{equation}
for every $\eta \in W^{1,G}_0(\mathcal D)$, almost every $\tau\in(t_1+h,t_2)$, and $h>0$ such that the functions are well defined. Similar results hold also for weak super- and subsolutions.

\begin{Proposition}\label{comparisonprinciple}
 \emph{(Comparison principle)} Let $\mathcal K := \mathcal D\times(t_1,t_2)\subset\Omega_T$ and let $u\in C^0({\overline {\mathcal K}}^p)$ be a weak subsolution to \eqref{general.equation}$_1$ and $v\in C^0({\overline {\mathcal K}}^p)$ a weak supersolution to \eqref{general.equation}$_1$ in $\mathcal K$. If $u\leq v$ on $\partial_p \mathcal K$, then $u\leq v$ in $\overline {\mathcal K}^p$.
\end{Proposition}
\begin{proof}
For $\eps>0$ fixed define $\varphi_\eps(t) := (t_2-\eps-t)_+$ and test \eqref{weaksolution} formally with 
\[
 \eta = (u_h-v_h-\eps)_+\varphi_\eps.
\]
Note that $\eta$ is compactly supported in $\mathcal K$ due to the continuity of $u$ and $v$ and the fact that $u\leq v$ on $\partial_pQ$. Subtracting the Steklov version of the variational inequality of $v$ from that of $u$ and integrating over $(t_1,t_2)$ yields 
\[
\int_{\mathcal K}\partial_t(u_h-v_h)\eta\dx\dt+\int_{\mathcal K}\langle[\A(Du)]_h-[\A(Dv)]_h, D\eta\rangle\dx\dt \leq 0.
\]
By the monotonicity of $\A$, Lemma~\ref{monotonicity}, we have 
\[
\begin{split}
\int_{\mathcal K}&\langle[\A(Du)]_h-[\A(Dv)]_h, D\eta\rangle\dx\dt\\
 &\to\int_{\mathcal K\cap\{u>v+\eps\}}\langle\A(Du)-\A(Dv), Du-Dv)\rangle\varphi_\eps\dx\dt\geq 0,
\end{split}
\]
and for the parabolic term we obtain using integration by parts 
\[
\begin{split}
\int_{\mathcal K}\partial_t(u_h-v_h)(u_h-v_h-\eps)_+\varphi_\eps\dx\dt &=-\frac12\int_{\mathcal K}(u_h-v_h-\eps)_+^2\partial_t\varphi_\eps\dx\dt\\
 &\to\frac12\int_{t_1}^{t_2-\eps}\int_{\mathcal D}(u-v-\eps)_+^2\dx\dt
\end{split}
\]
as $h\to0$. Combining these gives 
\[
 \int_{t_1}^{t_2-\eps}\int_{\mathcal D}(u-v-\eps)_+^2\dx\dt\leq 0,
\]
which implies $u\leq v+\eps$ almost everywhere in $\mathcal D\times(t_1,t_2-\eps)$. Since this holds for every $\eps>0$ and $u,v\in C(\overline Q)$, the result follows. 
\end{proof}

Observe that the uniqueness of a solution to the Cauchy-Dirichlet problem \eqref{general.equation} follows immediately from the previous result. Moreover, we have the following corollary. 

\begin{Corollary}\label{maximumprinciple}
 \emph{(Maximum principle)} Let $\mathcal K\subset\Omega_T$ and let $u\in C({\overline {\mathcal K}}^p)$ be a weak solution to \eqref{general.equation}$_1$ in $\mathcal K$. Then 
 \begin{equation*}%\label{max1}
  \inf_{\partial_p\mathcal K}u\leq u\leq\sup_{\partial_p\mathcal K}u
 \end{equation*}
in $\overline {\mathcal K}^p$ and, moreover, 
 \begin{equation*}%\label{max2}
  \sup_{\overline {\mathcal K}}|u|=\sup_{\partial_p\mathcal K}|u|.
 \end{equation*}
\end{Corollary}

We recall the following standard energy inequality for local weak solutions. We give it in a more general form for future reference.
\begin{Lemma}[Caccioppoli's inequality]\label{caccioppoli}
Let ${\mathcal K}:=\mathcal D\times(t_1,t_2)\Subset\Omega_T$ and let $u$ be a weak solution to \eqref{general.equation} in ${\mathcal K}$. Then there exists a constant $c\equiv c(g_0,g_1,\nu,L)$ such that 
 \begin{multline*}
  \sup_{\tau\in(t_1,t_2)}\int_{\mathcal D}\big[(u-k)_\pm^2\varphi^{g_1}\big](\cdot,\tau)\dx + \int_{\mathcal K}G\big(|D(u-k)_\pm|\big)\varphi^{g_1}\dx\dt\\
 \leq \int_{\mathcal D}\big[(u-k)_\pm^2\varphi^{g_1}\big](\cdot,t_1)\,dx+ c\int_{\mathcal K}\Big[G\big(|D\varphi|(u-k)_\pm\big)+(u-k)_\pm^2\left|\partial_t\varphi\right|\Big]\dx\dt
 \end{multline*}
for any $k\in\R$ and for every $\varphi \in W^{1,\infty}({\mathcal K})$ vanishing in a neighborhood of $\partial \mathcal D\times(t_1,t_2)$ and with $0\leq\varphi\leq 1$. The same inequality but only with the ``$+$'' sign holds for weak subsolutions.
\end{Lemma}
\begin{proof}
Fix $\varphi\in W^{1,\infty}({\mathcal K})$ as in the statement of the Lemma, call $w:=\pm(u-k)_\pm$ and choose $\eta = w_h\varphi^{g_1}$ as the test function in \eqref{Steklov.formulation}. Then we integrate over $(t_1,\tau)$ for $\tau\in(t_1,t_2)$ to obtain 
\begin{equation}\label{caccioppolisubsolution}
 \int_{\mathcal K}\partial_tu_h\,w_h\varphi^{g_1}\chi_{(t_1,\tau)}\dx\dt 
 + \int_{\mathcal K}\big\langle[\A(Du)]_h, D\left(w_h\varphi^{g_1}\right)\big\rangle\chi_{(t_1,\tau)}\dx\dt = 0.
\end{equation}
Integration by parts gives 
\begin{align}\label{parabolic}
 \int_{\mathcal K}\partial_tu_h\,w_h\varphi^{g_1}\chi_{(t_1,\tau)}\dx\dt 
  &=\frac{1}{2}\int_{t_1}^{\tau}\int_{\mathcal D}\partial_t(w_h^2)\varphi^{g_1}\dx\dt\\
  &=\left.\frac{1}{2}\int_{\mathcal D} w_h^2\varphi^{g_1}\,dx\right\rvert_{t=t_1}^{\tau}
  -\frac{1}{2}\int_{t_1}^{\tau}\int_{\mathcal D} w_h^2\partial_t\left(\varphi^{g_1}\right)\dx\dt\notag\\
    &\to\left.\frac{1}{2}\int_{\mathcal D} w^2\varphi^{g_1}\,dx\right\rvert_{t=t_1}^{\tau}
  -\frac{1}{2}\int_{t_1}^{\tau}\int_{\mathcal D} w^2\partial_t\left(\varphi^{g_1}\right)\dx\dt\notag
\end{align}
as $h\to0$. For the elliptic part we have by \eqref{A.ellipticity} 
\begin{align*}%\label{elliptic}
\int_{\mathcal K}&\big\langle[\A(Du)]_h, D(w_h\varphi^{g_1})\big\rangle \chi_{(t_1,\tau)}\dx\dt\\
&\to\int_{t_1}^{\tau}\int_{\mathcal D}\big\langle\A(Du), Dw\big\rangle\,\varphi^{g_1}\dx\dt+g_1\int_{t_1}^{\tau}\int_{\mathcal D}\big\langle\A(Du), D\varphi\big\rangle\,w\,\varphi^{g_1-1}\dx\dt\notag\\
&\geq c_1\int_{t_1}^{\tau}\int_{\mathcal D} G(|Dw|)\varphi^{g_1}\dx\dt - \,\bigg|g_1\int_{\mathcal K}\big\langle\A(Du), D\varphi\big\rangle\,w\,\varphi^{g_1-1}\dx\dt\bigg|,\notag
\end{align*}
where $c_1$ depends on $g_0,g_1,\nu$. Furthermore, by \eqref{A.bound}, Young's inequality with $\varepsilon\in (0,1)$ to be chosen and the properties of $g$ we obtain 
\begin{align}\label{elliptic2}
 &\left|g_1\int_{\mathcal K}\big\langle\A(Du), D\varphi\big\rangle\,w\,\varphi^{g_1-1}\dx\dt\right|
 \leq g_1\int_{\mathcal K}|\A(Dw)||D\varphi||w|\varphi^{g_1-1}\dx\dt\notag\\
 &\qquad\leq \varepsilon c_2 \int_{\mathcal K}\widetilde G\left(\frac{G(|Dw|)}{|Dw|}\varphi^{g_1-1}\right)\dx\dt
 + c(\varepsilon)\int_{\mathcal K}G(|D\varphi||w|)\dx\dt\notag\\
 &\qquad\leq \varepsilon c_2\int_{\mathcal K}G(|Dw|)\varphi^{g_1}\dx\dt
 + c(\varepsilon)\int_{\mathcal K}G(|D\varphi||w|)\dx\dt,
\end{align}
where $c_2$ depends on $g_0,g_1,L$ and $c(\eps)$ depends on $g_0,g_1,L$ as well as on $\eps$. Now, combining \eqref{parabolic}-\eqref{elliptic2} with \eqref{caccioppolisubsolution} yields
\[
\begin{split}
\left.\frac{1}{2}\int_{\mathcal D} w^2\varphi^{g_1}\,dx\right\rvert_{t=t_1}^{\tau}
-\frac{1}{2}\int_{\mathcal K}w^2\partial_t\left(\varphi^{g_1}\right)\dx\dt + c_1\int_{t_1}^{\tau}\int_{\mathcal D} G(|Dw|)\varphi^{g_1}\dx\dt \\
\leq \varepsilon c_2\int_{\mathcal K}G(|Dw|)\varphi^{g_1}\dx\dt + c(\varepsilon)\int_{\mathcal K}G(|D\varphi||w|)\dx\dt.
\end{split}
\]
We conclude by taking the essential supremum with respect to $\tau\in(t_1,t_2)$, choosing $\varepsilon\in(0,1)$ such that  $\varepsilon c_2\leq c_1/2$, reabsorbing the term on the right-hand side and recalling the definition of $w$.

The proof for subsolutions is very similar, taking into account that the test function $\eta$ must be nonnegative.
\end{proof}
\subsection{The geometry of the problem}\label{geometry}

In order to understand the equation, the first thing we want to stress is its scaling. Suppose $u$ solves the model equation \eqref{eq.model} in $Q_1=B_1\times(-1,0)$ and let $\kappa>0$. Then the function
\[
\bar u(x,t):= \kappa u\Big(\frac{x-x_0}{r},\frac{1}{\kappa^2}G\Big(\frac\kappa r\Big)(t-t_0)\Big)
\]
solves in 
\[
Q_r^\kappa(x_0,t_0):=B_r(x_0)\times\Big(t_0-\kappa^2\Big[G\Big(\frac\kappa r\Big)\Big]^{-1},t_0\Big)
\]
the equation 
\begin{equation}\label{eq.model.2}
\bar u_t -\divergence \Big( \frac{\bar g(|D\bar u|)}{|D\bar u|}D\bar u\Big)=0,
\end{equation}
where 
\begin{equation}\label{bar.g}
\bar g(s):=\frac\kappa r\Big[G\Big(\frac\kappa r\Big)\Big]^{-1}g\Big(\frac\kappa rs\Big). 
\end{equation}
The function $\bar g$ has the same structure as $g$, in the sense that it satisfies \eqref{O.property} exactly with parameters $g_0$ and 
$g_1$ and moreover, we have $\overline G(1)=1$, where 
\[
 \overline G(s):=\int_0^s\bar g(\sigma)\,d\sigma = \Big[G\Big(\frac\kappa r\Big)\Big]^{-1}G\Big(\frac{\kappa}rs\Big).
\]
Conversely, if we have a solution $w$ to \eqref{general.equation} in $Q_r^\kappa$, then 
\[
\bar w(x,t):=\frac1\kappa w\Big(x_0+rx, t_0+\kappa^2\Big[G\Big(\frac\kappa r\Big)\Big]^{-1}t\Big)
\]
solves \eqref{eq.model.2} in $Q_1$ with $\bar g$ as in \eqref{bar.g}.  In case we consider the general equation \eqref{general.equation}, the same scaling argument holds if we consider the vector field 
\[
\overline \A(\xi):=\frac\kappa r\Big[G\Big(\frac\kappa r\Big)\Big]^{-1}\A\Big(\frac\kappa r\xi\Big)
\]
which satisfies the structural conditions \eqref{assumptionsO} with $g$ replaced by the function $\bar g$.

\subsection{Other auxiliary results}
The following Lemma encodes the self-improving property of reverse H\"older inequalities. We take the form proposed in \cite[Lemma~5.1]{KMibe} with slight changes in order to meet our purposes.
\begin{Lemma}\label{KMlemma}
Let $\mu$ be a nonnegative Borel measure with finite total mass. Moreover, let $\gamma>1$ and 
$\{\sigma {\mathcal Q}\}_{0<\sigma\leq 1}$ be a family of open sets with the property
\[
\sigma'{\mathcal Q}\subset \sigma{\mathcal Q}\subset1{\mathcal Q}={\mathcal Q} 
\]
whenever $0<\sigma'<\sigma\leq 1$. If $w\in L^2({\mathcal Q})$ is a nonnegative function satisfying
\[
 \biggl(\int_{\sigma'{\mathcal Q}}w^{2\gamma}\,d\mu\biggr)^{1/(2\gamma)}
 \leq\frac{c_0}{\sigma-\sigma'} \biggl(\int_{\sigma{\mathcal Q}}w^2\,d\mu\biggr)^{1/2}
\]
for all $1/2\leq\sigma'<\sigma\leq 1$, then for any $0<q<2$ there is a positive constant $c\equiv c(c_0,\gamma,q)$ such that
\[
 \biggl(\int_{\sigma {\mathcal Q}}w^{2\gamma}\,d\mu\biggr)^{1/(2\gamma)}\leq\frac{c}{(1-\sigma)^{\xi}} \left(\int_{{\mathcal Q}}w^q\,d\mu\right)^{1/q},
\]
for all $0<\sigma<1$, where $\xi:=\frac{2\gamma-q}{q(\gamma-1)}$.
\end{Lemma}

The next one is a classic iteration Lemma. 
\begin{Lemma}\label{iteration}
Let $\phi:[R,2R]\to[0,\infty)$ be a function such that
\[
	\phi(r)\leq\frac12\phi(s)+\frac{A}{(s-r)^\beta}+B
	\qquad\text{for every}\ R\leq r<s\leq 2R,
\]
where $A,B\geq1$ and $\beta>0$. Then
\[
	\phi(R)\leq c(\beta)\, \bigg[\frac{A}{R^\beta}+B\bigg].
\]
 
\end{Lemma}

\section{A priori Lipschitz estimates}

In this section we impose on $u$ an additional regularity assumption and prove intrinsic estimates for the gradient of $u$. To be precise, we shall suppose
\begin{equation}\label{add.ass}
u,Du\in C^0_\loc(\Omega_T),\qquad u\in L^2_\loc(0,T;W^{2,2}_\loc(\Omega)).
\end{equation}
This is to say, we shall prove the estimates of this section as {\em a priori estimates}, leaving to Section \ref{approximation} 
the approximation procedure which will explain how to deduce the desired estimates without the additional assumption \eqref{add.ass}. Notice that the continuity of $u$ and $Du$ allows us to treat their pointwise values. Due to the assumed extra regularity it will be possible to differentiate the equation; this will be done by showing that the function 
\begin{equation}\label{v}
v:=|Du|^2 
\end{equation}
is a subsolution to a similar equation.

\begin{Lemma}\label{vsub}
Let $u$ be a weak solution to \eqref{general.equation}$_1$ in $\Omega_T$ and, moreover, assume that the regularity assumptions \eqref{add.ass} hold. Then $v$ is a weak subsolution to 
\begin{equation}\label{sub.differentiated}
  \partial_t v - \divergence\left(D\A(Du)Dv\right) = 0\qquad\text{in $\Omega_T$}.
\end{equation}
\end{Lemma}
\begin{proof}
Formally, the idea is to differentiate equation \eqref{general.equation}$_1$ with respect to $x_j$ for $j=1,\ldots,n$, then multiply by 
$D_ju$, and finally sum over $j$. To this end, let $0\leq\varphi\in C^{\infty}_c(\Omega_T)$, and test \eqref{weaksolution} with 
\[
 \eta = -D_j\left(D_ju\varphi\right).
\]
This choice can be justified by using Steklov averages, as done previously in the paper; we shall proceed formally. Integration by parts yields 
 \begin{align*}
  0&=-\int_{\Omega_T}u\,\partial_t\left(-D_j\left(D_ju\varphi\right)\right)\dx\dt 
  + \int_{\Omega_T}\langle\A(Du), D\left(-D_j\left(D_ju\varphi\right)\right)\rangle\dx\dt\\
  &=\int_{\Omega_T} \partial_t(D_ju)D_ju\varphi\dx\dt
  + \int_{\Omega_T}\langle D_j\A(Du), D\left(D_ju\varphi\right)\rangle\dx\dt\\
  &=-\frac{1}{2}\int_{\Omega_T} |D_ju|^2\partial_t\varphi\dx\dt
  +\frac{1}{2}\int_{\Omega_T}\langle D\A(Du)D\big(|D_ju|^2\big), D\varphi\rangle\dx\dt\\
  &\quad+\int_{\Omega_T}\langle D\A(Du)DD_ju, DD_ju\rangle\varphi\dx\dt.
 \end{align*}
Now, since 
\[
 \int_{\Omega_T}\langle D\A(Du)DD_ju, DD_ju\rangle\varphi\dx\dt\geq\nu\int_{\Omega_T}\frac{g(|Du|)}{|Du|}|DD_ju|^2\varphi\dx\dt\geq 0
\]
by $\eqref{assumptionsO}_1$, summing up over $j=1,\ldots,n$ leads to 
\[
 -\int_{\Omega_T} |Du|^2\partial_t\varphi\dx\dt+\int_{\Omega_T}\langle D\A(Du)D|Du|^2, D\varphi\rangle\dx\dt\leq 0.
\]
This proves the claim.
\end{proof}
Next we prove a Caccioppoli inequality of porous medium type for the function $v$.

\begin{Lemma}\label{caccioppoli2}
 Let $u$ be a weak solution of \eqref{general.equation} in $\Omega_T$ and assume that \eqref{add.ass} holds. Let $\mathcal K := \mathcal D\times (t_1,t_2) \Subset \Omega_T$ and $k\in\R$. Then there exists a constant $c\equiv c(\nu,L)$ such that 
 \begin{multline*}
  \sup_{\tau\in(t_1,t_2)}\int_{\mathcal D}[(v-k)_+^2\varphi^2](\cdot,\tau)\dx+\int_{{\mathcal K}}\frac{g(|Du|)}{|Du|}|D(v-k)_+|^2\varphi^2\dx\dt\\
  \qquad\leq c\int_{{\mathcal K}}(v-k)_+^2\left(\frac{g(|Du|)}{|Du|}|D\varphi|^2+\left|\partial_t\varphi\right|\right)\dx\dt
 \end{multline*}
\end{Lemma}
for every $\varphi\in C^{\infty}(\mathcal K)$ vanishing in a neighborhood of $\partial_p\mathcal K$. 
\begin{proof}
We can take 
\[
 \eta = (v-k)_+\varphi^2\chi_{(t_1,\tau)}
\]
for $\tau\in(t_1,t_2)$ as the test function in the weak formulation of \eqref{sub.differentiated}, up to a regularization similar to the previous ones. For the parabolic part we have 
\begin{multline*}
-\int_{t_1}^{\tau}\int_{\mathcal D} v\,\partial_t\left((v-k)_+\varphi^2\right)\dx\dt
=\frac{1}{2}\int_{t_1}^{\tau}\int_{\mathcal D}\partial_t(v-k)_+^2\varphi^2\dx\dt\\
=\frac{1}{2}\int_{\mathcal D}[(v-k)_+^2\varphi^2](\cdot,\tau)\dx
-\frac{1}{2}\int_{t_1}^{\tau}\int_{\mathcal D}(v-k)_+^2\partial_t\varphi^2\dx\dt.
\end{multline*}
The elliptic term can be estimated from below by using the assumptions \eqref{assumptionsO} and Young's inequality with 
$\varepsilon=\nu/(2L)$. This gives 
 \begin{align*}
\int_{t_1}^{\tau}\int_{\mathcal D}&\big\langle D\A(Du)Dv, D\left((v-k)_+\varphi^2\right)\big\rangle\dx\dt\\
 &=\int_{t_1}^{\tau}\int_{\mathcal D}\langle D\A(Du)D(v-k)_+, D(v-k)_+\rangle\varphi^2\dx\dt\\
 &\qquad+2\int_{t_1}^{\tau}\int_{\mathcal D}\langle D\A(Du)D(v-k)_+, D\varphi\rangle\,(v-k)_+\varphi\dx\dt\\
 &\geq\nu\int_{t_1}^{\tau}\int_{\mathcal D}\frac{g(|Du|)}{|Du|}|D(v-k)_+|^2\varphi^2\dx\dt\\
 &\qquad-2L\int_{t_1}^{\tau}\int_{\mathcal D}\frac{g(|Du|)}{|Du|}|D(v-k)_+||D\varphi|\,(v-k)_+\varphi\dx\dt\\
 &\geq\frac{\nu}{2}\int_{t_1}^{\tau}\int_{\mathcal D}\frac{g(|Du|)}{|Du|}|D(v-k)_+|^2\varphi^2\dx\dt\\
 &\qquad-c(\nu,L)\int_{t_1}^{\tau}\int_{\mathcal D}\frac{g(|Du|)}{|Du|}|D\varphi|^2\,(v-k)_+^2\dx\dt,
 \end{align*}
and thus, we obtain 
 \begin{multline*}
\int_{\mathcal D}[(v-k)_+^2\varphi^2](\cdot,\tau)\dx
+\nu\int_{t_1}^{\tau}\int_{\mathcal D}\frac{g(|Du|)}{|Du|}|D(v-k)_+|^2\varphi^2\dx\dt\\
\leq c\,\int_{{\mathcal K}}\frac{g(|Du|)}{|Du|}|D\varphi|^2\,(v-k)_+^2\dx\dt
+\int_{{\mathcal K}}(v-k)_+^2\left|\partial_t\varphi\right|\dx\dt.
\end{multline*}
Since $\tau\in(t_1,t_2)$ was arbitrary, the result follows.
\end{proof}

Combining the previous lemma with Sobolev's inequality leads to the following estimate.

\begin{Lemma}\label{reverseholder}
Let the assumptions of Lemma~\ref{caccioppoli2} be in force. Then there exists a constant 
$c\equiv c(n,g_1,\nu,L)$ such that 
\begin{multline}\label{gammaCaccioppoli}
 \mint_{\mathcal K}\frac{g(|Du|)}{|Du|}(v-k)_+^{2\gamma}\varphi^{2\gamma}\dx\dt\\
 \leq c\,|\mathcal D|^{2/n}(t_2-t_1)^{\gamma-1}\left(\mint_{\mathcal K}(v-k)_+^2
 \left(\frac{g(|Du|)}{|Du|}|D\varphi|^2+\left|\partial_t\varphi\right|\right)\dx\dt\right)^{\gamma},
\end{multline}
where - recall \eqref{convention} - 
\[
 \gamma:=2-\frac2{2^*}>1.
\]
\end{Lemma}
\begin{proof}
By H\"older's and Sobolev's inequalities we have 
\begin{align}\label{reverseholdereq1}
  \mint_{\mathcal K}&\frac{g(|Du|)}{|Du|}(v-k)_+^{2\gamma}\varphi^{2\gamma}\dx\dt\notag\\
  &=\frac{1}{t_2-t_1}\int_{t_1}^{t_2}\mint_{\mathcal D}\frac{g(|Du|)}{|Du|}(v-k)_+^{2}\varphi^{2}
  \big((v-k)_+^{2}\varphi^{2}\big)^{1-2/2^*}\dx\dt\notag\\
  &\leq\frac{1}{t_2-t_1}\int_{t_1}^{t_2}\biggl(\mint_{\mathcal D}
  \biggl(\Big(\frac{g(|Du|)}{|Du|}\Big)^{1/2}(v-k)_+\varphi\biggr)^{2^*}\dx\biggr)^{2/2^*}\notag\\
  &\qquad\qquad\qquad\qquad\qquad\qquad\qquad\qquad\times\biggl(\mint_{\mathcal D}(v-k)_+^2\varphi^2\dx\biggr)^{1-2/2^*}\dt\notag\\
  &\leq c(n)|\mathcal D|^{2/n}\biggl(\sup_{\tau\in(t_1,t_2)}\mint_{\mathcal D}[(v-k)_+^2\varphi^2](\cdot,\tau)\dx\biggr)^{1-2/2^*}\notag\\
  &\qquad\qquad\qquad\qquad\qquad\qquad\times
  \mint_{\mathcal K}\biggl|D\biggl(\Big(\frac{g(|Du|)}{|Du|}\Big)^{1/2}(v-k)_+\varphi\biggr)\biggr|^2\dx\dt.
\end{align}
A straightforward calculation yields 
\begin{align*}
  &\biggl|D\biggl(\Big(\frac{g(|Du|)}{|Du|}\Big)^{1/2}(v-k)_+\varphi\biggr)\biggr|^2\\
  &\qquad=\biggl|\biggl[\frac{(v-k)_+}{4v}\left(\frac{|Du|g'(|Du|)}{g(|Du|)}-1\right)+1\biggr]
  \Big(\frac{g(|Du|)}{|Du|}\Big)^{1/2}D(v-k)_+\varphi\\
  &\qquad\qquad\qquad\qquad\qquad\qquad\qquad\qquad\qquad\qquad
  +\Big(\frac{g(|Du|)}{|Du|}\Big)^{1/2}(v-k)_+D\varphi\biggr|^2\\
  &\qquad\leq c(g_1)\frac{g(|Du|)}{|Du|}\left|D(v-k)_+\right|^2\varphi^2
  +2\,\frac{g(|Du|)}{|Du|}(v-k)_+^2|D\varphi|^2,
\end{align*}
and thus, integrating and estimating the first term using Lemma~\ref{caccioppoli2} yields
\begin{multline*}
  \mint_{\mathcal K}\biggl|D\biggl(\Big(\frac{g(|Du|)}{|Du|}\Big)^{\frac{1}{2}}(v-k)_+\varphi\biggr)\biggr|^2\dx\dt\\
  \leq c\mint_{\mathcal K}(v-k)_+^2\left(\frac{g(|Du|)}{|Du|}|D\varphi|^2+\left|\partial_t\varphi\right|\right)\dx\dt,
\end{multline*}
where the constant $c$ depends only on $g_1,\nu$, and $L$. From Lemma~\ref{caccioppoli2} it also follows that 
\begin{multline*}
\sup_{\tau\in(t_1,t_2)}\mint_{\mathcal D}[(v-k)_+^2\varphi^2](\cdot,\tau)\dx\\
\leq c\,(t_2-t_1)\mint_{\mathcal K}(v-k)_+^2\left(\frac{g(|Du|)}{|Du|}|D\varphi|^2+\left|\partial_t\varphi\right|\right)\dx\dt;
\end{multline*}
therefore, by inserting the previous two inequalities into \eqref{reverseholdereq1} we obtain \eqref{gammaCaccioppoli}.
% \begin{align*}
%   \mint_{Q}&\frac{g(|Du|)}{|Du|}(v-k)_+^{2\gamma}\varphi^{2\gamma}\dx\dt\\
%   &\qquad\leq c\,|B|^{2/n}(t_2-t_1)^{\gamma-1}
%   \left(\mint_{Q}(v-k)_+^2\left(\frac{g(|Du|)}{|Du|}|D\varphi|^2+\left|\partial_t\varphi\right|\right)\dx\dt\right)^{\gamma},
% \end{align*}
% where $c\equiv c(n,g_1,\nu,L)$. 
 \end{proof}

Next the aim is to prove an intrinsic reverse H\"older's inequality. To this end, let $Q_\rho(x_0,t_0)\subset \Omega_T$, let $\lambda\geq 1$ be such that
\begin{equation}\label{lambdaineq}
\lambda\geq \frac14 \sup_{Q_\rho(x_0,t_0)}|Du|,
\end{equation}
and set 
\[
 \theta_{\lambda} := \frac{g(\lambda)}{\lambda}. 
\]
We introduce the intrinsic cylinder 
\begin{equation}\label{cyl.int}
Q_{\rho}^{\lambda}\equiv Q_{\rho}^{\lambda}(x_0,t_0) 
:= \min\{1,\theta_{\lambda}\}^{1/2}B_{\rho}(x_0)\times\left(t_0-\min\{1, \theta_{\lambda}^{-1}\}\rho^2,t_0\right). 
\end{equation}
Note that we have the alternative expression
\[
Q_{\rho}^{\lambda} =
\begin{cases}
   B_{\rho}(x_0)\times\left(t_0-\theta_{\lambda}^{-1}\rho^2,t_0\right), \qquad& \theta_{\lambda} \geq 1\\[4mm]
   \theta_{\lambda}^{1/2}B_{\rho}(x_0)\times\left(t_0-\rho^2,t_0\right), & 0<\theta_{\lambda} < 1,\\[1mm]
\end{cases}
\]
from which we easily see the analogy with the intrinsic geometry used to handle the parabolic $p$-Laplacian, recalling that in this case 
$g(s)/s=s^{p-2}$ and $\lambda$ is ``dimensionally comparable'' to $|Du|$. Observe that we clearly have $Q_{\rho}^{\lambda}(x_0,t_0)\subset Q_{\rho}(x_0,t_0)$ in any case.

\begin{Lemma}\label{reverseholder2}
 Let $u$ be a weak solution to \eqref{general.equation}$_1$ in $\Omega_T$, assume that \eqref{add.ass} and \eqref{lambdaineq} hold and let $q>0$. Then there exists a constant $c\equiv c(n,g_1,\nu,L,q)$ such that 
 \[
  \biggl(\mint_{Q_{\rho/2}^{\lambda}}(v-k)_+^{2\gamma}\dx\dt\biggr)^{1/(2\gamma)} \leq c\biggl(\mint_{Q_{\rho}^{\lambda}}(v-k)_+^q\dx\dt\biggr)^{1/q}
 \]
for every  $k\geq\lambda^2$.
\end{Lemma}
\begin{proof}
Let $1/2\leq\sigma'<\sigma\leq 1$ and choose a cut-off function $\varphi\in C^{\infty}(\sigma Q_{\rho}^{\lambda})$ vanishing in the neighborhood of $\partial_p (\sigma Q_{\rho}^{\lambda})$ such that $0\leq\varphi\leq 1$, $\varphi = 1$ in $\sigma' Q_{\rho}^{\lambda}$, and 
\[
 |D\varphi|\leq\frac{c}{\rho(\sigma-\sigma')}\min\big\{1, \theta_{\lambda}\big\}^{-1/2}, \qquad 
 \left|\partial_t\varphi\right|\leq\frac{c}{\rho^2(\sigma-\sigma')^2}\min\big\{1, \theta_{\lambda}^{-1}\big\}^{-1}.
\]
Observe that by the inclusion $Q_{\rho}^{\lambda}(x_0,t_0)\subset Q_{\rho}(x_0,t_0)$ and \eqref{lambdaineq} we have
\[
 |Du| \leq  4\lambda\qquad\text{in $Q_{\rho}^{\lambda}$}.
\]
Moreover, we have $|Du| \geq \lambda$  in the support of $(v-k)_+$, since $k\geq\lambda^2$ and $v=|Du|^2$. Thus, by using the properties of $g$ we obtain 
\begin{equation}\label{rehomog}
\frac14\theta_{\lambda}\leq\frac{g(|Du|)}{|Du|}\leq c(g_1)\theta_{\lambda} 
\end{equation}
in $Q_{\rho}^{\lambda}\cap\{v\geq k\}$.
Now Lemma~\ref{reverseholder} yields 
\begin{align*}
  \mint_{\sigma'Q_{\rho}^{\lambda}}&(v-k)_+^{2\gamma}\dx\dt\leq c(n)\, \theta_{\lambda}^{-1}
  \mint_{\sigma Q_{\rho}^{\lambda}}\frac{g(|Du|)}{|Du|}(v-k)_+^{2\gamma}\varphi^{2\gamma}\dx\dt\\
  &\leq c\, \theta_{\lambda}^{-1}\Big|\min\{1, \theta_{\lambda}\}^{1/2}B_{\sigma\rho}\Big|^{2/n}
  \Big(\min\{1, \theta_{\lambda}^{-1}\}(\sigma\rho)^2\Big)^{\gamma-1}\\
  &\qquad\times\biggl(\mint_{\sigma Q_{\rho}^{\lambda}}(v-k)_+^2
  \biggl(\frac{g(|Du|)}{|Du|}|D\varphi|^2+|\partial_t\varphi|\biggr)\dx\dt\biggr)^{\gamma}\\
  &\leq c\, \theta_{\lambda}^{-1}\min\{1, \theta_{\lambda}\}\min\{1, \theta_{\lambda}^{-1}\}^{\gamma-1}\rho^{2\gamma}\\
  &\qquad\times\biggl(\frac{ \theta_{\lambda}\min\{1, \theta_{\lambda}\}^{-1}
  +\min\{1, \theta_{\lambda}^{-1}\}^{-1}}{\rho^2(\sigma-\sigma')^2}
  \mint_{\sigma Q_{\rho}^{\lambda}}(v-k)_+^2\dx\dt\biggr)^{\gamma}\\
  &= \frac{c}{(\sigma-\sigma')^{2\gamma}}\biggl(\mint_{\sigma Q_{\rho}^{\lambda}}(v-k)_+^2\dx\dt\biggr)^{\gamma}.
\end{align*}
This is to say
\[
\biggl(\mint_{\sigma'Q_{\rho}^{\lambda}}(v-k)_+^{2\gamma}\dx\dt\biggr)^{1/(2\gamma)}
 \leq\frac{c}{\sigma-\sigma'}\biggl(\mint_{\sigma Q_{\rho}^{\lambda}}(v-k)_+^2\dx\dt\biggr)^{1/2},
\]
where the constant $c$ depends only on $n,g_1,\nu,L$. 

Next we use Lemma~\ref{KMlemma} with $w=(v-k)_+$ and 
$d\mu=\frac{1}{|Q_{\rho}^{\lambda}|}\dx\dt$. This gives for every $0<q<2$ a constant $c\equiv c(n,g_1,\nu,L,q)$ such that 
\[
 \biggl(\mint_{Q_{\rho/2}^{\lambda}}(v-k)_+^{2\gamma}\dx\dt\biggr)^{1/(2\gamma)} \leq c\biggl(\mint_{Q_{\rho}^{\lambda}}(v-k)_+^q\dx\dt\biggr)^{1/q};
\]
the case $q\geq2$ now follows from H\"older's inequality.
\end{proof}

Iterating the previous result yields the following pointwise estimate.

\begin{Proposition}\label{weakmax}
 Let $u$ be a weak solution to \eqref{general.equation} in $\Omega_T$ and assume that \eqref{add.ass} holds. Then for every $q>0$ there exists a constant $c\equiv c(n,g_1,\nu,L,q)$ such that 
 \[
  |Du(x_0,t_0)|\leq \lambda+c\,\biggl(\mint_{Q_\rho^{\lambda}(x_0,t_0)} \big(|Du|^2-\lambda^2\big)_+^q\dx\dt\biggr)^{1/(2q)}
 \]
holds for every $\lambda$ satisfying \eqref{lambdaineq}.
\end{Proposition}
\begin{proof}
 The idea is to apply De Giorgi's iteration method with the aid of Lemma~\ref{reverseholder2}. Let us first consider the case $0<q<2$. To this end, choose for $j\in\N_0$ 
\[
 \rho_j = 2^{-j}\rho, \qquad k_j = \lambda^2 + (1-2^{-j})d,
\]
where $d>0$ is to be determined later. Observe that $\rho_0=\rho$, $k_0=\lambda^2$, and $\rho_j$ decreases to zero and $k_j$ increases 
to $\lambda^2+d$ as $j$ tends to infinity; clearly $k_j\geq\lambda^2$. Denote $Q_j:=Q_{\rho_j}^{\lambda}(x_0,t_0)$ and 
\[
 Y_j:=\biggl(\mint_{Q_j}(v-k_j)_+^q\dx\dt\biggr)^{1/q}\qquad\text{for $j\in\N_0$. }
\]
By Lemma~\ref{reverseholder2} we have 
\[
 \biggl(\mint_{Q_{j+1}}(v-k_j)_+^{2\gamma}\dx\dt\biggr)^{1/(2\gamma)}\leq c\biggl(\mint_{Q_j}(v-k_j)_+^q\dx\dt\biggr)^{1/q},
\]
and since $k_{j+1}>k_j$ implies 
\[
 (v-k_j)_+^{2\gamma} \geq (k_{j+1}-k_j)^{2\gamma-q}(v-k_{j+1})_+^q\chi_{\{v\geq k_{j+1}\}},
\]
we obtain  
\[
 Y_{j+1}\leq \frac{c}{(k_{j+1}-k_j)^{\beta}}\biggl(\mint_{Q_{j+1}}(v-k_j)_+^{2\gamma}\dx\dt\biggr)^{2\gamma/q}\leq c^\ast d^{-\beta} 2^{\beta j}Y_j^{1+\beta},
\]
 for every $j\in\N_0$, where $\beta:=2\gamma/q-1>0$ and $c^\ast\equiv c^\ast(n,g_1,\nu,L,q)$. Then a standard hyper-geometric iteration lemma implies $Y_j\to 0$ as $j\to \infty$, provided that 
 \[
 Y_0\leq (2c^\ast)^{-\frac1\beta}d
 \]
 and this can be guaranteed by choosing 
 \[
  d=(2c^\ast)^{\frac1\beta}\biggl(\mint_{Q_\rho^{\lambda}(x_0,t_0)} \big(v-\lambda^2\big)_+^q\dx\dt\biggr)^{1/q}.
 \]
%Next we show by induction that 
%\[
% Y_j \leq \frac{Y_0}{2^j}
%\]
%for every $j\in\N_0$. Clearly the claim holds when $j=0$. Assume then that it holds for some $j$. By \eqref{weakmaxeq1} and the 
%definition of $k_j$ we obtain 
%\[
%\begin{split}
% Y_{j+1}&\leq c\,\frac{2^{(j+1)\beta}}{d^{\beta}}\left(\frac{Y_0}{2^j}\right)^{1+\beta}
% =c\,\left(\frac{Y_0}{d}\right)^{\beta}\frac{Y_0}{2^{j+1}},
%\end{split}
%\]
%and therefore, by choosing $d=c^{\frac{1}{\beta}}Y_0$, we see that the claim holds also for $j+1$. 
Now Lebesgue's differentiation theorem yields 
\begin{align*}
 \big(v(x_0,t_0)-\big(\lambda^2+d\big)\big)_+=\lim_{j\to\infty}\biggl(\mint_{Q_j}\big(v-\big(\lambda^2+d\big)\big)_+^q\dx\dt\biggr)^{1/q}
 \leq \lim_{j\to\infty}Y_j=0, 
\end{align*}
which implies, recalling the choice of $d$, 
\[
 v(x_0,t_0)\leq \lambda^2 + c\biggl(\mint_{Q_\rho^{\lambda}(x_0,t_0)}\big(v-\lambda^2\big)_+^q\dx\dt\biggr)^{1/q}.
\]
The case $q\geq2$ follows again by H\"older's inequality.
\end{proof}

\section{Approximation}\label{approximation}

In this section we regularize the equation in order to apply the results of the previous section and show that the gradient of the solution to the regularized equation is uniformly bounded. Then all we have left to prove is that the approximating solutions converge to a function that solves the original equation.

\vs

To this end, define for $\eps\in(0,1)$ 
\begin{equation}\label{regularization.A}
\A_\eps(\xi):=(\phi_\eps\ast\A)(\xi)+\eps \big(1+|\xi|\big)^{\widetilde{g}_1-2}\xi, 
\end{equation}
where $\phi_\eps(\xi)=\phi({\xi/\eps})/{\eps^n}$; $\phi$ is a standard mollifier with $\int_{\R^n}\phi\dx=1$. 
That is, we mollify the vector field $\A$ and perturb it with the nondegenerate $\widetilde{g}_1$-Laplacian, where $\widetilde{g}_1>\max\{g_1,2\}$; we can take for example $\widetilde g_1 := g_1+1$. It is straightforward to see that $\A_\eps$ satisfies \eqref{assumptionsO} with $g$ replaced by
\begin{equation}\label{geps}
g_\eps(s):=\frac{g(s+\eps)}{s+\eps}s+\eps (1+s)^{\widetilde{g}_1-2}s 
\end{equation}
and $L,\nu$ replaced by $\widetilde L=c(n,g_1)L,\widetilde\nu={\nu}/{c(n,g_1)}$, see also Paragraph \ref{weak.assumptions}. Now the key point is that $\mathcal O_{g_\eps}$ can be bounded {\em independently of} $\eps$. Indeed, we have 
% \[
% \mathcal O_{g_\eps}(s)=\frac{g'(s+\eps)\frac{s^2}{s+\eps}+g(s+\eps)\frac{\eps s}{(s+\eps)^2}+\eps (1+s)^{\widetilde{g}_1-2}s
% +(\widetilde{g}_1-2)\eps(1+s)^{\widetilde{g}_1-3}s^2}{g(s+\eps)\frac{s}{s+\eps}+\eps (1+s)^{\widetilde{g}_1-2}s}.
% \]
% Using \eqref{O.property} and discarding the last term in the numerator we have 
% \[
% \mathcal O_{g_\eps}(s)\geq \frac{(g_0-1)g(s+\eps)\big(\frac{s}{s+\eps}\big)^2+g(s+\eps)\frac{\eps s}{(s+\eps)^2}+\eps(1+s)^{\widetilde{g}_1-2}s}{g(s+\eps)\frac{s}{s+\eps}+\eps (1+s)^{\widetilde{g}_1-2}s}\geq\min\{g_0-1,1\},
% \]
% while we can bound from above
% \[
% \mathcal O_{g_\eps}(s)\leq \frac{(g_1-1)g(s+\eps)\big(\frac{s}{s+\eps}\big)^2+g(s+\eps)\frac{\eps s}{(s+\eps)^2}+(\widetilde{g}_1-1)\eps(1+s)^{\widetilde{g}_1-2}s}{g(s)+\eps (1+s)^{\widetilde{g}_1-2}s}
% \leq \widetilde{g}_1-1,
% \]
% since $\widetilde{g}_1 > \max\{g_1,2\}$. Hence, all in all we have
\[
\widetilde{g}_0-1\leq \mathcal O_{g_\eps}(s)\leq\widetilde{g}_1-1,
\]
where $\widetilde{g}_0:= \min\{g_0,2\}$. Note that $g_\eps$ also satisfies the lower bound in \eqref{grow.below.g}, since $g_\eps(s)\geq g(s)/2$ for $s\geq 1$. 

\vs

Let $u_\eps\in V^{2,\widetilde g_1}(\Omega_T)\cap C^0(\overline\Omega_T)$ be the solution to the Cauchy-Dirichlet problem
\begin{equation}\label{approx.eq}
\begin{cases}
 \partial_tu_\eps-\divergence \A_\eps(Du_\eps)=0\qquad&\text{in $\Omega_T$,}\\[3pt]
 u_\eps=\psi&\text{on $\partial_p \Omega_T$;}
\end{cases}
\end{equation}
for existence and uniqueness of such solutions see for instance \cite{LK}.
Since
\[
\eps (1+s)^{\widetilde{g}_1-2} \leq \frac{g_\eps(s)}{s}\leq \frac{c(g_1)}{\eps}(1+s)^{\widetilde{g}_1-2},
\]
in addition to satisfying $g_\eps$-ellipticity and -growth conditions analogous to \eqref{assumptionsO}, the vector field $\A_\eps$ 
also enjoys {\em nondegenerate} $p$-Laplacian growth conditions with $p=\widetilde{g}_1$. Hence, by standard theory, $u_\eps$ satisfies the assumption \eqref{add.ass}, see \cite{DiBenedetto, KMibe}; therefore the results of the previous section are at our disposal for $u\equiv u_\eps$. Note that all the constants will turn out to be effectively {\em independent of} $\eps$.

\vs

Let us then show how to apply the result of the previous section in order to locally bound the gradient of the approximating solution uniformly in terms of $\eps$. Here we also prove an estimate that, once convergence is established, leads to \eqref{loc.lipschitz}. Observe that the assumption \eqref{grow.below.g} is crucial in this proof. We shall shorten $\|\psi\|_{L^\infty}\equiv\|\psi\|_{L^\infty(\partial_p\Omega_T)}$. 

\begin{Proposition}\label{uniform.gradient.bound}
Let $u_\eps$ be a solution to \eqref{approx.eq} and let $\mathcal K\Subset\Omega_T$. Then $\|Du_\eps\|_{L^\infty(\mathcal K)}$ is bounded by a constant depending on $\data,\epsilon,c_\ell,\|\psi\|_{L^\infty}$, and $\dist_{\rm par}(\partial_p\Omega_T,\mathcal K)$, but independent of $\eps$.
\end{Proposition}

\begin{proof}
Let us consider a standard parabolic cylinder $Q_{4R}\equiv Q_{4R}(x^*,t^*)\subset\Omega_T$ and  a subcylinder $Q_\rho(x_0,t_0)\subset Q_{2R}$. Moreover, let $\lambda\geq 1$ be such that 
\begin{equation}\label{lambdaineq2}
 \lambda\geq \frac{1}{4} \sup_{Q_\rho(x_0,t_0)}|Du_\eps|.
\end{equation}
We divide the proof into two cases depending on which term of $g_\eps$ dominates at $\lambda$. 

\subsubsection*{Case I} Assume 
\[
\frac{g(\lambda+\eps)}{\lambda+\eps}\leq\eps(1+\lambda)^{\widetilde g_1-2}.
\]
Setting 
\[
 \theta_{\lambda}^\eps := \frac{g_\eps(\lambda)}{\lambda} =\frac{g(\lambda+\eps)}{\lambda+\eps}+\eps(1+\lambda)^{\widetilde g_1-2}
\]
we clearly have 
\begin{equation}\label{almost.trivial}
 \eps(1+\lambda)^{\widetilde g_1-2}\leq \theta_{\lambda}^\eps
 \leq 2\eps(1+\lambda)^{\widetilde g_1-2}. 
\end{equation}
By applying Proposition~\ref{weakmax} to $u_\eps$ with $q=\widetilde g_1/2$ we obtain 
\begin{align*}
 |Du_\eps(x_0,t_0)|
 &\leq\lambda+c\biggl(\mint_{Q_\rho^{\lambda}(x_0,t_0)} \left(|Du_\eps|^2-\lambda^2\right)_+^{\widetilde g_1/2}\dx\dt\biggr)^{1/\widetilde g_1}\\
 &\leq\lambda+c\biggl(\frac{\max\{1, \theta_{\lambda}^\eps\}}{\min\{1, \theta_{\lambda}^\eps\}^{n/2}}
 \mint_{Q_\rho(x_0,t_0)} |Du_\eps|^{\widetilde g_1}\dx\dt\biggr)^{1/\widetilde g_1},
\end{align*}
since $Q_\rho^{\lambda}(x_0,t_0)\subset Q_\rho(x_0,t_0)$. 

We further distinguish two cases: in the case when $ \theta_{\lambda}^\eps\geq 1$ we get 
\[
 \frac{\max\{1, \theta_{\lambda}^\eps\}}{\min\{1, \theta_{\lambda}^\eps\}^{n/2}} = \theta_{\lambda}^\eps
 \leq 2\eps(1+\lambda)^{\widetilde g_1-2},
\]
while when $0< \theta_{\lambda}^\eps<1$ we have 
\begin{equation}\label{plug.into}
\frac{\max\{1, \theta_{\lambda}^\eps\}}{\min\{1, \theta_{\lambda}^\eps\}^{n/2}} = ( \theta_{\lambda}^\eps)^{-n/2}\leq \left(\eps(1+\lambda)^{\widetilde g_1-2}\right)^{-n/2}=\eps\left(\eps^{1+2/n}(1+\lambda)^{\widetilde g_1-2}\right)^{-n/2};
\end{equation}
in both cases we have used \eqref{almost.trivial}. Since 
\begin{align*}
\eps&\geq \frac{g(\lambda+\eps)}{(1+\lambda)^{\widetilde g_1-2}(\lambda+\eps)}\\
&\geq c_\ell \big(\lambda+\eps\big)^{\frac{n-2}{n+2}+\epsilon-1}\big(1+\lambda\big)^{2-\widetilde g_1}\\
&\geq c_\ell \big(1+\lambda\big)^{2-\widetilde g_1+\min\{\epsilon-4/(n+2),0\}}=:c_\ell \big(1+\lambda\big)^{\bar \eta}
\end{align*}
by \eqref{grow.below.g} and the fact that $\lambda\geq 1$, plugging this estimate into \eqref{plug.into} yields
\begin{align*}
 \left(\eps^{1+2/n}\big(1+\lambda\big)^{\widetilde g_1-2}\right)^{-n/2}
 &\leq c(n,c_\ell)\big(1+\lambda\big)^{-(\bar\eta(1+2/n)+\widetilde g_1-2)n/2}\\
 &\leq c(n,c_\ell)\big(1+\lambda\big)^{\widetilde g_1-\min\{\epsilon(n+2)/2,2\}};
\end{align*}
a direct computation shows indeed the relation between the exponents. Hence we have 
\begin{align*}
 |Du_\eps(x_0,t_0)|&\leq\lambda+c\big(1+\lambda\big)^{1-\min\{\epsilon(n+2)/(2\widetilde g_1),2/\widetilde g_1\}}\\
 &\hspace{4cm}\times  \biggl(\eps\mint_{Q_\rho(x_0,t_0)} |Du_\eps|^{\widetilde g_1}\dx\dt\biggr)^{1/\widetilde g_1}\\
 &\leq 2\lambda+c\biggl(\eps\mint_{Q_\rho(x_0,t_0)} |Du_\eps|^{\widetilde g_1}\dx\dt\biggr)^{\max\left\{\frac{2}{\epsilon(n+2)},\frac12\right\}}+1
\end{align*}
by Young's inequality; we also used $\widetilde g_1>2$.

\subsubsection*{Case II} Suppose then that 
\[
\frac{g(\lambda+\eps)}{\lambda+\eps}>\eps(1+\lambda)^{\widetilde g_1-2}.
\]
Here we have 
\[
\frac{g(\lambda+\eps)}{\lambda+\eps}\leq \theta_{\lambda}^\eps
 \leq 2\,\frac{g(\lambda+\eps)}{\lambda+\eps},
\]
and again by Proposition~\ref{weakmax} 
\begin{align*}
 |Du_\eps(x_0,t_0)|\leq\lambda+c\biggl(\frac{\max\{1, \theta_{\lambda}^\eps\}}{\min\{1, \theta_{\lambda}^\eps\}^{n/2}}
 \mint_{Q_\rho(x_0,t_0)} \big(|Du_\eps|^2-\lambda^2\big)_+^{q}\dx\dt\biggr)^{\frac{1}{2q}}.
  \end{align*}
When $ \theta_{\lambda}^\eps\geq 1$, choosing $q=1$ leads to 
 \begin{align*}
  |Du_\eps(x_0,t_0)|&\leq\lambda+c\biggl(\frac{g(\lambda+\eps)}{\lambda+\eps}
 \mint_{Q_\rho(x_0,t_0)}\big(|Du_\eps|^2-\lambda^2\big)_+\dx\dt\biggr)^{\frac12}\\
 &\leq\lambda+c\biggl(\mint_{Q_\rho(x_0,t_0)}\frac{g(|Du_\eps|)}{|Du_\eps|}\big(|Du_\eps|^2-\lambda^2\big)_+\dx\dt\biggr)^{\frac12}\\
 &\leq\lambda+c\biggl(\mint_{Q_\rho(x_0,t_0)}G(|Du_\eps|)\chi_{\{|Du_\eps|\geq 1\}}\dx\dt\biggr)^{\frac12}.
 \end{align*}
The second inequality stems from the fact that 
\begin{equation}\label{bound.bilateral}
\frac12(\lambda+\eps)\leq |Du_\eps|\leq 4(\lambda+\eps) 
\end{equation}
in the set $Q_\rho(x_0,t_0)\cap\{|Du_\eps|\geq \lambda\}$ by \eqref{lambdaineq2}, while for the last one we used \eqref{O.property3} and the fact that $\lambda\geq1$. 

In the case $0< \theta_{\lambda}^\eps<1$ we choose $q=\epsilon(n+2)/4$ and use \eqref{grow.below.g} and again \eqref{bound.bilateral} to obtain 
 \begin{align*}
  |Du_\eps(x_0,t_0)|&\leq\lambda+c\biggl(\biggl(\frac{g(\lambda+\eps)}{\lambda+\eps}\biggr)^{-\frac n2}
 \mint_{Q_\rho(x_0,t_0)}\big(|Du_\eps|^2-\lambda^2\big)_+^q\dx\dt\biggr)^{\frac1{2q}}\\
 &\leq\lambda+c\biggl(\mint_{Q_\rho(x_0,t_0)}
 \biggl(\frac{|Du_\eps|}{g(|Du_\eps|)}\biggr)^{\frac n2}\big(|Du_\eps|^2-\lambda^2\big)_+^q\dx\dt\biggr)^{\frac1{2q}}\\
 &\leq\lambda+c\biggl(\mint_{Q_\rho(x_0,t_0)}
 |Du_\eps|^{\left(1-\frac{n-2}{n+2}-\epsilon\right)\frac n2+2q}\chi_{\{|Du_\eps|\geq 1\}}\dx\dt\biggr)^{\frac1{2q}}\\
 &=\lambda+c\biggl(\mint_{Q_\rho(x_0,t_0)}
 |Du_\eps|^{1+\frac{n-2}{n+2}+\epsilon}\chi_{\{|Du_\eps|\geq 1\}}\dx\dt\biggr)^{\frac2{\epsilon(n+2)}}\\
 &\leq\lambda+c\biggl(\mint_{Q_\rho(x_0,t_0)}G(|Du_\eps|)\chi_{\{|Du_\eps|\geq 1\}}\dx\dt\biggr)^{\frac{2}{\epsilon(n+2)}};
 \end{align*}
note that
\[
\Big(1-\frac{n-2}{n+2}-\epsilon\Big)\frac n2+2q=\Big(\frac{4}{n+2}-\epsilon\Big)\frac n2+\epsilon\frac{n+2}2=1+\frac{n-2}{n+2}+\epsilon.
\]
Therefore in both cases we have 
\[
 |Du_\eps(x_0,t_0)|
 \leq\lambda+c\bigg(\mint_{Q_\rho(x_0,t_0)}G(|Du_\eps|)\chi_{\{|Du_\eps|\geq 1\}}\dx\dt\bigg)^{\max\left\{\frac12,\frac{2}{\epsilon(n+2)}\right\}}.
\]
Combining Cases I and II and denoting $\tilde \eta:=\max\big\{\frac12,\frac{2}{\epsilon(n+2)}\big\}$ yields 
\begin{align}\label{pointwise.eps}
 &|Du_\eps(x_0,t_0)|\notag\\
 &\qquad\leq 2\lambda+c\,\bigg(\mint_{Q_\rho(x_0,t_0)}\left(G(|Du_\eps|)\chi_{\{|Du_\eps|\geq 1\}}+\eps|Du_\eps|^{\widetilde g_1}\bigg)\dx\dt\right)^{\tilde\eta}+1\notag\\
 &\qquad\leq 2\lambda+c\,\Big(\frac{R}{\rho}\Big)^{(n+2)\tilde\eta}\biggl(\mint_{Q_{2R}}G_\eps(|Du_\eps|)\dx\dt\biggr)^{\tilde\eta}+1,
\end{align}
since 
\[
 G(s) \leq \frac1{g_0}g(s+\eps)(s+\eps)\leq \frac4{g_0}\frac{g(s+\eps)}{s+\eps}s^2 \leq \frac4{g_0}sg_\eps(s) \leq \frac{4\widetilde g_1}{g_0}G_\eps(s)
\]
for $s\geq 1$ and trivially 
\[
 \eps s^{\widetilde g_1} \leq \eps(1+s)^{\widetilde g_1-2}s^2 \leq \widetilde g_1 G_\eps(s).
\]
The constant $c$ in \eqref{pointwise.eps} depends only on $\data$, $\epsilon,c_\ell$. 

\vs

Let us now choose two intermediate cylinders $Q_R\subset Q_r\Subset Q_s\subset Q_{2R}$ and fix
\[
 \lambda := 1 + \frac{1}{4}\|Du_\eps\|_{L^{\infty}(Q_s)} < \infty, \qquad (x_0,t_0)\in Q_r,\qquad \rho:=\frac{s-r}{2}>0.
\]
Clearly $Q_\rho(x_0,t_0)\subset Q_s$ so that \eqref{lambdaineq2} holds. Then \eqref{pointwise.eps} implies 
\begin{multline*}
\|Du_\eps\|_{L^{\infty}(Q_r)}\leq \frac12\|Du_\eps\|_{L^{\infty}(Q_s)}\\+c\,\Big(\frac{R}{s-r}\Big)^{(n+2)\tilde\eta}\biggl(\mint_{Q_{2R}}G_\eps(|Du_\eps|)\dx\dt\biggr)^{\tilde\eta}+3. 
\end{multline*}
Now, by choosing $\phi(r)=\|Du_\eps\|_{L^{\infty}(Q_r)}$, iteration Lemma \ref{iteration} gives 
\begin{equation}\label{Linfty1}
\|Du_\eps\|_{L^{\infty}(Q_R)}\leq c\biggl(\mint_{Q_{2R}}\big[G_\eps(|Du_\eps|)+1\big]\dx\dt\biggr)^{\tilde\eta}.
\end{equation}
At this point, in order to get rid of the dependence on $\eps$ on the right-hand side, the idea is to use the Caccioppoli inequality of Lemma~\ref{caccioppoli} to translate the dependence on $Du_\eps$ to one on $u_\eps$, and the latter in turn into a dependence on $\psi$. Indeed, take $\varphi\in C^\infty(Q_{4R})$ vanishing in a neighborhood of $\partial_p Q_{4R}$ such that $0\leq\varphi\leq 1$, $\varphi = 1$ in $Q_{2R}$, and $|D\varphi|^2+|\partial_t\varphi|\leq c/R^2$. Since 
\[
\sup_{Q_{4R}}|u_\eps|\leq \sup_{\partial_p \Omega_T}|u_\eps| = \sup_{\partial_p \Omega_T}|\psi| \leq \|\psi\|_{L^\infty} 
\]
by the maximum principle, Corollary~\ref{maximumprinciple}, we can estimate by Lemma~\ref{caccioppoli} 
\begin{align}\label{Linfty2}
 \mint_{Q_{2R}}G_\eps(|Du_\eps|)\dx\dt  
 &\leq c \mint_{Q_{4R}}\left[G_\eps(|D\varphi||u_\eps|)+u_\eps^2\left|\partial_t\varphi\right|\right]\dx\dt\notag\\
  &\leq c \bigg(1+\frac{\|\psi\|_{L^\infty}}{R}\bigg)^{\widetilde g_1} + c\left(\frac{\|\psi\|_{L^\infty}}{R}\right)^2\notag\\
  &= c\big(\data,\epsilon,c_\ell,\|\psi\|_{L^\infty},R\big).
 \end{align}
Note that the constant does not depend on $\eps$. Therefore we conclude the proof of the Proposition, modulo a standard covering argument.
\end{proof}

\subsection{A uniform interior modulus of continuity via Lipschitz regularity}

In this section we prove that the approximating solutions $u_\eps$ are equicontinuous in the interior of the domain; in particular we shall show their equi-Lipschitz regularity with respect to the parabolic metric.

\begin{Proposition}\label{int.unif.continuity}
Let $u_\eps$ be a solution to \eqref{approx.eq}. Then $u_\eps\in {\rm Lip}({1,1/2})(\Omega_T)$ locally, uniformly in $\eps$; this is to say, for every subcylinder $\mathcal K\Subset\Omega_T$ there exists a constant $c$ depending on $\data,\epsilon,c_\ell, \|\psi\|_{L^\infty}$, and $\dist_{\rm par}(\partial_p\Omega_T,\mathcal K)$ such that
\begin{equation}\label{almost.uniform.Lipschitz}
|u_\eps(x,t)-u_\eps(y,s)|\leq c\,\dist_{\rm par}\big((x,t),(y,s)\big)
\end{equation}
for every $(x,t),(y,s)\in\mathcal K$ and for every $\eps\in(0,1)$.
\end{Proposition}
\begin{proof}
Fix an intermediate set $\mathcal K'$ such that $\mathcal K\Subset \mathcal K'\Subset\Omega_T$ and 
\[
\dist_{\rm par}(\hat z,\partial_p\Omega_T)=\dist_{\rm par}(\mathcal K,\partial_p\Omega_T)/2=:d/2
\]
 for every $\hat z\in\partial_p\mathcal K'$. Take also a cylinder $Q_r(x_0,t_0)\subset\mathcal K'$ with $(x_0,t_0)\in\mathcal K$; this will happen for instance if $r\leq d/2$. Since $Du_\eps$ is continuous, by applying the divergence theorem and using the bound for $\A_\eps$ in \eqref{A.bound} we infer
\begin{align*}
\mint_{B_r(x_0)}u_\eps(\cdot,\tau)\dx\biggr|_{\tau=t_1}^{t_2}&=\frac nr\int_{t_1}^{t_2}\mint_{\partial {B_r(x_0)}}\Big\langle\A_\eps(Du_\eps),\frac{x-x_0}{|x-x_0|}\Big\rangle\,d\mathcal H^{n-1}\dt  \\
&\leq \frac cr\int_{t_1}^{t_2}\mint_{\partial {B_r(x_0)}}g_\eps(Du_\eps)\,d\mathcal H^{n-1}\dt
\end{align*}
for all $t_0-r^2< t_1\leq t_2<t_0$, where $\mathcal H^{n-1}$ stands for the $(n-1)$-dimensional Hausdorff measure.  We thus estimate
\begin{align*}%\label{temp.est}
\osc_{\tau\in(t_0-r^2,t_0)} (u_\eps)_{B_r(x_0)}(\tau)&=\sup_{t_0-r^2<t_1\leq t_2<t_0}\biggl|\mint_{B_r(x_0)}u_\eps(\cdot,\tau)\dx\biggr|_{\tau=t_1}^{t_2}\biggr|\\
&\leq\frac cr\int_{t_0-r^2}^{t_0}\mint_{\partial {B_r(x_0)}}g_\eps(Du_\eps)\,d\mathcal H^{n-1}\dt \notag\\
%&\leq c\,r\Big[g\big(\|Du_\eps\|_{L^\infty(\mathcal K')}+1\big)+\|Du_\eps\|_{L^\infty(\mathcal K')}^{\widetilde g_1-1}+1\Big]\notag\\
&\leq c\,r\left(1+\|Du_\eps\|_{L^\infty(Q_r(x_0,t_0))}\right)^{\widetilde g_1-1}\notag\\
&\leq c\,r\left(1+\|Du_\eps\|_{L^\infty(\mathcal K')}\right)^{\widetilde g_1-1}.\notag\notag
\end{align*}
Now by Proposition~\ref{uniform.gradient.bound}, in particular by \eqref{Linfty1}-\eqref{Linfty2}, we have
\begin{equation}\label{osc.slicewise}
\osc_{\tau\in(t_0-r^2,t_0)} (u_\eps)_{B_r(x_0)}(\tau)\leq c\big(\data, c_\ell,\epsilon,\|\psi\|_{L^\infty},d\big)\,r.
\end{equation}
At this point we simply split for $(x_1,t_1),(x_2,t_2)\in Q_r(x_0,t_0)$
\begin{multline*}
|u_\eps(x_1,t_1)-u_\eps(x_2,t_2)|\leq  \Big|u_\eps(x_1,t_1)-\mint_{B_r(x_0)}u_\eps(\cdot,t_1)\dx\Big|\\
+\bigg|\mint_{B_r(x_0)}u_\eps(\cdot,t_1)\dx-\mint_{B_r(x_0)}u_\eps(\cdot,t_2)\dx\bigg|+\Big|u_\eps(x_2,t_2)-\mint_{B_r(x_0)}u_\eps(\cdot,t_2)\dx\Big|.
\end{multline*}
While in order to bound the second term we shall use \eqref{osc.slicewise}, the first and last terms can be estimated using the mean value theorem as follows:
\begin{align*}
 \Big|u_\eps(x_i,t_i)-\mint_{B_r(x_0)}u_\eps(\cdot,t_i)\dx\Big|&\leq \mint_{B_r(x_0)}\big|u_\eps(x_i,t_i)-u_\eps(x,t_i)\big|\dx\\
 &\leq 2r\,\|Du_\eps\|_{L^\infty(\mathcal K')},
\end{align*}
for $i\in\{1,2\}$. Therefore, using again Proposition~\ref{uniform.gradient.bound}, we have 
\begin{equation}\label{almost.final.oscillation}
\osc_{Q_r(x_0,t_0)}u_\eps \leq c\,r
\end{equation}
with $c$ as in \eqref{osc.slicewise}, in particular {\em not depending} on $\eps$. To conclude the proof, for $(x_1,t_1)$, $(x_2,t_2)\in\mathcal K$, we simply check whether $\dist_{\rm par}\big((x_1,t_1),(x_2,t_2)\big)\leq d/4$ holds true or not; if so, then there exists a cylinder $Q_r(x_0,t_0)$ with $r=\dist_{\rm par}\big((x_1,t_1),(x_2,t_2)\big)$ such that $(x_1,t_1),(x_2,t_2)\in Q_r(x_0,t_0)$ and we can apply \eqref{almost.final.oscillation} that directly yields \eqref{almost.uniform.Lipschitz}. If on the other hand $\dist_{\rm par}\big((x_1,t_1),(x_2,t_2)\big)> d/4$, then, again simply using the maximum principle, we have
\begin{equation}\label{last.modification}
|u_\eps(x,t)-u_\eps(y,s)|\leq 2\|u_\eps\|_{L^\infty(\Omega_T)}\leq 8\,\frac{\dist_{\rm par}\big((x_1,t_1),(x_2,t_2)\big)}{d}\|\psi\|_{L^\infty}; 
\end{equation}
the proof is concluded.
\end{proof}

\begin{Remark}\label{rem.2}
Notice that, tracking the dependence on $d$ of the constant in Proposition \ref{int.unif.continuity} and in turn the dependence on $R$ of estimate \eqref{Linfty2}, and also slightly modifying the previous proof, we deduce that  estimate \eqref{almost.uniform.Lipschitz} can be rewritten as 
\begin{equation}\label{improved.est.interior}
|u_\eps(x,t)-u_\eps(y,s)|\leq \frac{c}{d_{z,w}^\gamma}\,\dist_{\rm par}\big((x,t),(y,s)\big), 
\end{equation}
for an exponent $\gamma\equiv \gamma(n,g_1,\epsilon)\geq1$ and a constant $c$ depending only on $\data,\epsilon,c_\ell$, $\|\psi\|_{L^\infty}$, with $z=(x,t), w=(y,s)$ and accordingly
\[
d_{z,w}:=\min\big\{\dist_{\rm par}(z,\partial_p\Omega_T),\dist_{\rm par}(w,\partial_p\Omega_T),1\big\}. 
\]
Indeed, if $\dist_{\rm par}(z,w)\leq d_{z,w}/8$, then we can apply the argument in the first part of the proof of Proposition \ref{int.unif.continuity} with $r=\dist_{\rm par}(z,w)$ to get (suppose $s\leq t$)
\[
|u_\eps(z)-u_\eps(w)|\leq\osc_{Q_r(z)}u_\eps \leq \frac c{d_{z,w}^\gamma}\,\dist_{\rm par}(z,w), 
\]
where $\gamma=\widetilde g_1(\widetilde g_1-1)\tilde\eta$, since we have $Q_r(z)\subset Q_{d_{z,w}/8}(z)$, $Q_{d_{z,w}/2}(z)\subset \Omega_T$ and so
\[
\|Du_\eps\|_{L^\infty(Q_r(z))}\leq \|Du_\eps\|_{L^\infty(Q_{d_{z,w}/8}(z))}\leq \frac{c}{d_{z,w}^{\widetilde g_1\tilde\eta}}
\]
The case where $d_{z,w}<8\,\dist_{\rm par}(z,w)$ can be approached exactly as in \eqref{last.modification}.
\end{Remark}

\section{Continuity at the  boundary}\label{boundary.continuity}

In this section we prove that the solution to the approximating problem \eqref{approx.eq} is continuous up to the boundary independently of $\eps$ by building an explicit barrier. We do not want to enter the details of the theory and the general relation between existence of barriers and regularity of the boundary points; the interested reader can see the nice paper \cite{LK} for the evolutionary $p$-Laplacian, while \cite{HKM, Lqnotes} summarize the results in the elliptic setting. 

We shall begin with the proof of the continuity at the lateral boundary; here we shall give all the details needed. For the continuity at the initial boundary we shall however only sketch the proof, which on the other hand is very similar and easier than the lateral case. Again, we will prove the existence of a uniform (in the sense that it will be independent of $\eps$) modulus of continuity for $u_\eps$; in the last section we shall show that this modulus is easily inherited by the limit of $u_\eps$.

\vs

Let us begin with the construction of an explicit barrier at the lateral boundary. Due to a scaling argument that will be clear soon it is enough to consider a very special case. 

\subsection{An explicit construction of a supersolution at the boundary}\label{barrier.paragraph}
We define the function 
\[
 v^+(x,t):=|x'|^2+M\sqrt{x_n}+(2t+1)_-,
\]
where $M\geq 1$ is to be chosen depending on $\data$. We aim to show that $v^+$ is a weak supersolution in 
\[
\mathcal  Q:=\big\{(x,t)\in\R^{n+1}:|x'|\leq 1, \ x_n\in[0,2], \  t\in[-1,0]\big\}.
\]
Simple calculations show that 
\[
\begin{split}
 Dv^+=(2x',M x_n^{-1/2}/2),\qquad \partial_tv^+=-2\chi_{\{-1<t<-1/2\}},\\
 D^2v^+={\rm diag}\,(2,\ldots,2,-M x_n^{-3/2}/4),\qquad
\end{split}
\]
and moreover, since $D^2_{i,j}v^+=0$ whenever $i\neq j$, we have 
\begin{align*}
 \divergence\A(Dv^+)&=\sum_{i=1}^nD_i\A_i(Dv^+)=\sum_{i,j=1}^nD_{\xi_j}\A_i(Dv^+)D^2_{i,j}v^+\\
 &=2\sum_{i=1}^{n-1}D_{\xi_i}\A_i(Dv^+) - \frac M4 D_{\xi_n}\A_n(Dv^+) x_n^{-3/2}.
\end{align*}
The first term we estimate from above using $\eqref{assumptionsO}_2$ and for the second term we can apply $\eqref{assumptionsO}_1$, since $D_{\xi_n}\A_n(Dv^+)=\langle D\A(Dv^+)\hat e_n,\hat e_n\rangle$. Furthermore, if we require $M\geq 2^{3/2}16(n-1)L/\nu$, we obtain 
\begin{align}\label{super.computation}
 \divergence\A(Dv^+)&\leq \Big(2(n-1)L-\frac\nu4 M x_n^{-3/2}\Big)\frac{g(|Dv^+|)}{|Dv^+|}\\\notag
 &\leq -\frac\nu8M x_n^{-3/2}\frac{g(|Dv^+|)}{|Dv^+|}.
\end{align}
Now, observe that since $M\geq 4$ we also get 
\[
 |Dv^+|=\sqrt{4|x'|^2+\big(M x_n^{-1/2}/2\big)^2}\leq Mx_n^{-1/2} 
\]
in $\mathcal Q$. On the other hand, we have 
\[
 |Dv^+|\geq Mx_n^{-1/2}/2\geq 1.
\]
Using these estimates we obtain 
\[
 \frac{g(|Dv^+|)}{|Dv^+|}\geq |Dv^+|^{g_0-2} \geq 
 \begin{cases}
1,&g_0\geq 2 \\[3pt]
\big(Mx_n^{-1/2}\big)^{g_0-2} &g_0<2
\end{cases},
\]
and thus 
\[
 \divergence\A(Dv^+)\leq -\frac{\nu}{8} M^{\min\{g_0,2\}-1}x_n^{-(\min\{g_0,2\}+1)/2}.
\]
The exponent of $x_n$ is negative, so that by choosing $M\equiv M(\data)$ large enough (recall that $g_0>1$), we finally obtain 
\[
 \partial_t v^+-\divergence\A(Dv^+)\geq -2+\frac{\nu}{8}2^{-(\min\{g_0,2\}+1)/2}M^{\min\{g_0,2\}-1} \geq 0. 
\]
It is easy to see that $v^+\in V_{\textrm{loc}}^{2,G}(\mathcal Q)$ and thus $v^+$ is a (weak) supersolution in $\mathcal Q$.

\subsection{A reduction of the oscillation in a significant case}\label{red.particolar}

We set ourselves now in what seems to be a very particular, unitary case; it will be clear soon that, up to a simple rescaling procedure, this will be the significant case for the proof.

\vs

Let $\bar\Omega$ be a bounded $C^{1,\beta}$ domain and $\bar\Omega_T:=\bar\Omega\times(-1,0)$. Suppose that $0\in\partial{\bar \Omega}$ and the orthonormal system where the boundary is a graph is the standard cartesian one, with the direction where $\partial {\bar \Omega}$ is a graph given by $\hat e_n$. We hence have
\[
\partial{\bar \Omega}\cap \{|x'|<1,|x_n|<1\}={\rm graph}\,{\bar\theta},\quad\text{ with \ \ ${\bar\theta}:B_1'(0)\to(-1,1)$\ \  and\ \  ${\bar\theta}(0)=0$}
\]
and ${\bar \Omega}\cap \{|x'|<1,|x_n|<1\}$ is the epigraph of ${\bar\theta}$. Let $\bar u$ be a weak solution to \eqref{general.equation}$_1$ in 
\[
\bar\Omega_T\cap\mathcal Q_1 \qquad \text{with}\qquad \mathcal Q_1:=B'_1\times(-1,1)\times(-1,0)\subset\R^{n+1},
\]
such that $\bar u=\bar \psi$ in $\partial_p\bar\Omega_T\cap\mathcal Q_1$. Moreover, we suppose $\bar \psi(0)=\bar u(0)=0$. Take $\delta\in(0,1)$ to be fixed later. We assume that
\begin{equation}\label{small.boundary}
\text{the graph of ${\bar\theta}$ over $B'_1$ is contained in the cylinder $B'_1\times(-\delta,\delta)$}
\end{equation}
and moreover that 
\begin{equation}\label{osc.psi}
\osc_{\partial{\bar \Omega}\cap B'_1\times(-1,1)}\bar\psi\leq\delta  \qquad \text{and }\qquad \osc_{{\bar\Omega_T}\cap \mathcal Q_1}{\bar u}\leq 1
\end{equation}

Let us take the barrier $v_+$ built in the previous paragraph and shift it in the $\hat e_n$ direction as follows:
\[
v_\delta^+(x',x_n,t):=v^+(x',x_n+\delta,t)+\delta. 
\]
Now $v_\delta^+$ is defined and continuous, in particular,  over the parabolic closure of ${\bar \Omega}_T\cap \mathcal Q_\delta$, where $\mathcal Q_\delta=B_1'\times(-\delta,1)\times  (-1,0)$, and there it is still a supersolution to an equation structurally similar to \eqref{general.equation}$_1$. The aim is to prove that $\bar u \leq v_\delta^+$ on $\partial_p(\bar\Omega_T\cap \mathcal Q_\delta)$ by considering the different pieces:
\begin{itemize}
\item on $[\partial{\bar \Omega}\times(-1,0)]\cap\mathcal Q_\delta$ we estimate
\[
{\bar u}-v_\delta^+ \leq \bar\psi-\delta\leq0
\]
using \eqref{osc.psi}$_1$ and since $v^+\geq0$;
\item on $\big[(\partial B_1'\times[-\delta,1])\cap\overline{\bar \Omega}\big]\times(-1,0)$ we have
\[
{\bar u}-v_\delta^+ \leq 1-1-\delta\leq0,
\]
by \eqref{osc.psi}$_2$ together with  ${\bar u}(0)=0$ and the fact that $v^+\geq1$, since $|x'|=1$;
\item on $\overline{B_1'}\times\{1\}\times(-1,0)$ we have
\begin{equation}\label{perte}
{\bar u}-v_\delta^+ \leq 1- M\leq0, 
\end{equation}
since ${\bar u}\leq1$ as above and on $\{x_n=1\}$ we have $v_\delta^+\geq M\geq1$;
\item finally, on $\overline{{\bar \Omega}\cap(B_1'\times(-\delta,1))}\times\{-1\}$ we again have $v_\delta^+\geq1$ due to the expression of the time-dependent part, and therefore the conclusion again follows.
\end{itemize}
Note that the first three pieces exhaust the lateral boundary of ${\bar \Omega}_T\cap\mathcal Q_\delta$, while the fourth one makes up its initial boundary. Therefore, we have ${\bar u}\leq v_\delta^+$ on the parabolic boundary of ${\bar \Omega}_T\cap\mathcal Q_\delta$ and hence, by Proposition \ref{comparisonprinciple}, ${\bar u}\leq v_\delta^+$ in ${\bar \Omega}_T\cap\mathcal Q_\delta$. Now, if $\delta\leq1/2$, we have
\[
v^+_\delta\leq \delta^2+M(2\delta)^{1/2}\qquad\text{ in $[(B'_\delta\times(-\delta,\delta))\cap{\bar \Omega}]\times(-\delta,0)$.}
\]
 Therefore, if we choose $\delta$ small enough, depending only on $M$ and so ultimately on $\data$, such that $\delta^2+M(2\delta)^{1/2}\leq 1/4$, then we have
\[
\sup_{[(B'_\delta\times(-\delta,\delta))\cap{\bar \Omega}]\times(-\delta,0)}{\bar u}\leq \frac14.
\]
Completely analogously we may consider the subsolution $v_-(x',x_n,t)=-v_+(x',x_n,t)$ to obtain a corresponding bound from below. All in all, we conclude with
\begin{equation}\label{red.osc.onestep}
\osc_{[B_\delta\cap{\bar \Omega}]\times(-\delta,0)}{\bar u}\leq \frac12.
\end{equation}
 
\subsection{Iteration}

Let $R_0\leq \min\{R_\Omega,1\}$ be fixed and let $Q_r^\omega(x_0,t_0)$ be a cylinder not intersecting the initial boundary, with $x_0\in\partial\Omega$, $\omega>0$ and $r\leq R_0$. Since we are supposing $R_0\leq R_\Omega$, we have that the boundary of $\Omega$ can be written as a $C^{1,\beta}$ graph in $B_r$: there exists a unitary vector $\hat e\in\R^n$ such that if we set $T:\R^n\to \R^n$ for the orthogonal transformation that maps $\hat e_n=(0,\dots,0,1)$ into $\hat e$, we have
\[
T^{-1}(\partial\Omega-x_0)\cap \big(B'_r\times(-r,r)\big)={\rm graph}\,\theta
\]
for some $\theta\in C^{1,\beta}(B_r')$ with values in $(-r,r)$. We now start from the assumption
\begin{equation}\label{initial.iteration}
\osc_{\Omega_T\cap Q_r^\omega(x_0,t_0)}u_\eps\leq\omega.
\end{equation}
We define, for $j\in\N$, the quantities
\[
\omega_j:=2^{-j}\omega,\qquad r_{j+1}=\min\big\{\sigma r_j,\overline r_{j+1}\big\},\qquad r_0=r
\]
where $\sigma\in(0,1/2)$ is such that $\sigma\leq\frac\delta{\sqrt2}$ and $(2\sigma)^{g_0}\leq 4(\sqrt2)^{-g_1}\delta$ (see \eqref{inclusion...}), with $\delta\in(0,1/2)$ being the constant defined in the previous paragraph, and $\bar r_j$ is such that
\begin{equation}\label{osc.psi.eps}
\osc_{\partial_{\rm lat}\Omega_T\cap Q_{\overline r_j }^{\omega_j}(x_0,t_0)}\psi\leq \delta\,\omega_j. 
\end{equation}
Note that this is possible, since $\psi$ is continuous so that at $\omega_j$ fixed the map $\rho\mapsto \osc_{Q_\rho^{\omega_j}(x_0,t_0)}\psi$ vanishes as $\rho\to0$. We prove by induction
\begin{equation}\label{inductive.iteration}
\osc_{\Omega_T\cap Q_{r_j}^{\omega_j}(x_0,t_0)}u_\eps\leq\omega_j.
\end{equation}
Now \eqref{inductive.iteration}$_0$ is simply \eqref{initial.iteration}, so we suppose that \eqref{inductive.iteration}$_j$ holds and we prove \eqref{inductive.iteration}$_{j+1}$, for $j\in\N_0$. Rescale $u_\eps$ as follows:
\[
\bar u(x,t):=\frac1\omega_ju_\eps\Big(x_0+\frac{r_j}{\sqrt2}Tx,t_0+\omega_j^2\Big[G\Big(\frac{\sqrt2\omega_j}{r_j}\Big)\Big]^{-1}t\Big)-\frac{u_\eps(x_0,t_0)}{\omega_j}.
\]
This is a solution to an equation structurally similar to \eqref{general.equation}$_1$, see Paragraph \ref{geometry}, in particular in $[(B'_1\times(-1,1))\cap\Omega]\times(-1,0)$, with boundary datum
\[
\bar \psi(x,t):=\frac1\omega_j\psi\Big(x_0+\frac{r_j}{\sqrt2}Tx,t_0+\omega_j^2\Big[G\Big(\frac{\sqrt2\omega_j}{r_j}\Big)\Big]^{-1}t\Big)-\frac{u_\eps(x_0,t_0)}{\omega_j}
\]
and where the boundary of $\bar\Omega:=[\sqrt 2T^{-1}(\Omega-x_0)/r_j]\cap (B'_1\times(-1,1))$ is given by the graph of the function $\bar\theta(x')=\theta(r_jx'/\sqrt 2)/r_j$ over $B'_1$. We have
\[
\osc_{B'_1}|D\bar\theta|=\frac1{\sqrt2}\osc_{B'_{r_j/\sqrt2}}|D\theta|\leq \frac{1}{\sqrt2}\osc_{B'_{R_0}}|D\theta|\leq \frac{R_0^\beta}{\sqrt2}\Theta.
\]
Now we choose $R_0$ small enough so that the right hand side of the chain of inequalities in the above display is smaller than $\delta$, where $\delta$ is the quantity fixed in the previous paragraph. This ensures that \eqref{small.boundary} is satisfied (since $D\theta(0)=D\bar\theta(0)=0=\theta(0)=\bar \theta(0)$). Since all the other assumptions in Paragraph \ref{red.particolar} are satisfied (in particular by our choice of $\bar r_j$), we have estimate \eqref{red.osc.onestep} at hand; therefore \eqref{inductive.iteration}$_{j+1}$ follows by our definition of $r_{j+1}$ and $\omega_{j+1}$. Indeed, scaling back we have
\[
\osc_{\Omega_T\cap \hat{\mathcal Q}_{r_j}}u_\eps\leq \frac12\omega_j\quad \text{with\quad$\hat{\mathcal Q}_{r_j}:=B_{\delta r_j/\sqrt2}(x_0)\times\Big(t_0-\delta\,\omega_j^2\Big[G\Big(\frac{\sqrt2\omega_j}{r_j}\Big)\Big]^{-1},t_0\Big)$}
\]
and by \eqref{GDelta2} and our definition of $\sigma$, we infer
\begin{multline}\label{inclusion...}
\omega_{j+1}^2\Big[G\Big(\frac{\omega_{j+1}}{r_{j+1}}\Big)\Big]^{-1}\leq \frac{(2\sigma)^{g_0}}4\omega_j^2\Big[G\Big(\frac{\omega_j}{r_j}\Big)\Big]^{-1}\\
 \leq \frac\delta{(\sqrt2)^{g_1}}\omega_j^2\Big[G\Big(\frac{\omega_j}{r_j}\Big)\Big]^{-1}\leq \delta\,\omega_j^2\Big[G\Big(\frac{\sqrt2\omega_j}{r_j}\Big)\Big]^{-1}.
\end{multline}
Finally, we note that the lengths of the time intervals also go to zero, that is, the cylinders are shrinking. Indeed, the first inequality in the above computation shows that the ratio of two consecutive time scales is bounded by $(2\sigma)^{g_0}/4$, which is clearly strictly smaller than one.

\subsection{Some quantitative estimates}\label{quantify}
Let us set
\[
\omega:=2\|\psi\|_{L^\infty}+1,
\] 
fix a radius $r<R_0$, and take a point $(x_0,t_0)\in\partial_{\rm lat}\Omega_T$ such that $Q^G_{\max\{1,\omega^{2/g_0-1}\}r}(x_0,t_0)$ does not intersect the initial boundary. Clearly \eqref{initial.iteration} holds by the maximum principle. Now we recall that $\psi$ has the modulus of continuity $\omega_\psi$:
\[
|\psi(x,t)-\psi(y,s)|\leq\omega_\psi\big(\dist_{{\rm par},G}((x,t),(y,s))\big)
\]
for all $(x,t),(y,s)\in\partial_p\Omega_T$. Since 
%\[
%Q_{\rho}^{\omega_j}(x_0,t_0)\subset 
%\begin{cases}
%B_\rho(x_0)\times (t_0-[G(\omega^{1-2/g_0}/\rho)]^{-1},t_0)\  &\text{if $g_1<2$},\\[3mm]
%B_\rho(x_0)\times (t_0-[G(2^{-j(1-2/g_1)}\omega^{1-2/g_0}/\rho)]^{-1},t_0)\  &\text{if $g_1\geq 2$},
%\end{cases}
%\]
\begin{equation*}%\label{inclusion.above}
Q_{\rho}^{\omega_j}(x_0,t_0)\subset  B_\rho(x_0)\times (t_0-[G(2^{-j(1-2/g_1)}\omega^{1-2/g_0}/\rho)]^{-1},t_0), 
\end{equation*}
we have
%\[
%Q_{\rho}^{\omega_j}(x_0,t_0)\subset 
%\begin{cases}
%Q_{A\rho}(x_0,t_0) &\text{if $g_1<2$},\\[2mm]
%Q_{2^{j(1-2/g_1)}A\rho}(x_0,t_0) \qquad&\text{if $g_1\geq 2$},
%\end{cases}
%\]
\[
Q_{\rho}^{\omega_j}(x_0,t_0)\subset Q^G_{A_j\rho}(x_0,t_0),
\]
with $A_j:=\max\{1,2^{j(1-2/g_1)}\omega^{2/g_0-1}\}\geq1$. Thus we see that if we want \eqref{osc.psi.eps} satisfied, it is enough to require
\[
A_j\bar r_j\leq \max\{1,\omega^{2/g_0-1}\}r,
\]
so that $Q^G_{A_j\bar r_j}(x_0,t_0)$ does not intersect the initial boundary, and
\[
 \omega_\psi(\bar r_j)\leq \delta\, A_j^{-1}\omega_j=\delta\min\{2^{-j}\omega,2^{-2j(1-1/g_1)}\omega^{2(1-1/g_0)}\}
\]
by the concavity of $\omega_\psi(\cdot)$. At this point we have \eqref{inductive.iteration} at our disposal, and this will be used noting that in particular we have
\begin{equation*}%\label{inclusion.below}
Q_{r_j}^{\omega_j}(x_0,t_0)\supset Q^G_{r_j/B_j}(x_0,t_0) 
\end{equation*}
with $B_j:=\max\{1,2^{-j(1-2/g_0)}\omega^{1-2/g_1}\}\geq1$. Hence, for $(x,t)\in \Omega_T\cap Q_{r/B_0}^G(x_0,t_0)$ fixed we find the largest $j\in\N_0$ such that 
\[
\frac{r_{j+1}}{B_{j+1}}\leq \dist_{{\rm par},G}((x,t),(x_0,t_0))<\frac{r_j}{B_j}.
\]
Note that this is possible, since clearly $r_j/B_j\leq r_j \to0$ as $j\to\infty$. At this point 
\[
|u_\eps(x,t)-u_\eps(x_0,t_0)|\leq \osc_{\Omega_T\cap Q_{r_j}^{\omega_j}(x_0,t_0)}u_\eps\leq 2^{-j}\omega.
\]
Let $\{(r_{j+1}/B_{j+1},2^{-j}\omega)\}_{j\in\N_0}$ be a sequence of points in $\R^2$ and call $\omega_u$ the smallest concave function such that $\omega_u(r_{j+1}/B_{j+1})\geq2^{-j}\omega$; note that $\omega_u$ is a modulus of continuity. For instance, one can take the piecewise linear interpolation of the sequence $\{(x_j,y_j)\}_{j\in\N}$ given by $x_j=\max_{k\geq j+1}r_k/B_k, y_j=2^{-j}\omega$, which is component-wise decreasing as $j$ increases. This finally leads to
\begin{equation}\label{boundary.modulus}
|u_\eps(x,t)-u_\eps(x_0,t_0)|\leq  {2^{-j}}\omega\leq \omega_u(r_{j+1}/B_{j+1})\leq \omega_u\big(\dist_{{\rm par},G}((x,t),(x_0,t_0))\big),
\end{equation}
and this holds for $(x,t)\in \Omega_T\cap Q_{r/B_0}^G(x_0,t_0)$. In fact, it also holds for points $(x,t)$ outside $Q_{r/B_0}^G(x_0,t_0)$, since then we have $\dist_{{\rm par},G}((x,t),(x_0,t_0))>r/B_0$ and thus 
\[
 |u_\eps(x,t)-u_\eps(x_0,t_0)|\leq \frac{2B_0}{r}||\psi||_{L^{\infty}} \frac{r}{B_0}\leq c \,\dist_{{\rm par},G}((x,t),(x_0,t_0))
\]
by the maximum principle. Note that the modulus of continuity $\omega_u$ at this point depends on $\data,||\psi||_{L^{\infty}},\omega_\psi$ but also on $r$. 

If now $\psi$ is $\gamma$-H\"older continuous with respect to the $G$-parabolic metric, then we see that it is enough to take $\bar r_j = c(\data,\omega,\gamma)2^{-\eta j}r$ for some $\eta\equiv\eta(g_1,\gamma)$. This yields that the numbers $r_j$ can be written as $\bar\eta^{j}r$ for some $\bar\eta\in(0,1)$. Now the H\"older continuity follows, for instance, similarly to \cite[Chapter III, Lemma 3.1]{DiBenedetto}. 

\subsection{Continuity at the initial boundary}
We begin by modifying the barrier built in Paragraph \ref{barrier.paragraph} to meet the different situations at the initial boundary. We start by considering the case where, before rescaling, we have a solution in a cylinder $B_r(x_0)\times(0,\omega^2/G(\omega/r))$, with $B_r(x_0)\subset\Omega$, equal to $\bar \psi$ over $B_r(x_0)\times\{0\}$; that is, the true case of initial boundary continuity. Later on we shall face the ``corner case'', that is the case of cylinders $B_r(x_0)\times(0,\omega^2/G(\omega/r))$ with $x_0\in\partial\Omega$.

\vs

After rescaling, one sees that it is enough to build a supersolution in $\mathcal Q:=B_1\times(0,1)$. In this case the explicit expression is simply $v^+(x,t):=|x|^{1/2}$. We then have $v^+\in V^{2,G}_{\textrm{loc}}(\mathcal Q)$ and $v^+$ is a supersolution to \eqref{general.equation}$_1$ in $\mathcal Q$. Moreover, if we further suppose that
\[
\osc_{B_1\times\{0\}}\bar\psi\leq\delta,\qquad\osc_{\mathcal Q}{\bar u}\leq 1, \qquad \bar u(0,0)=\bar \psi(0,0)=0
\]
for some $\delta\in(0,1)$, it is easy to see that ${\bar u}\leq v^++\delta$ on $ \partial_p\mathcal Q$. Indeed on $\partial B_1\times(0,1)$ we have $v^+=1$ but $\bar  u\leq1$, while on $\overline{B_1}\times\{0\}$ we have $\bar u=\bar\psi\leq \delta$ and $v^+\geq0$. Therefore we can deduce by Proposition \ref{comparisonprinciple} that ${\bar u}\leq v^++\delta$ in $\mathcal Q$. Now the proof goes on similarly as in Paragraphs \ref{red.particolar} to \ref{quantify}, with possibly new constants $\delta$ and $R_0$. 

\vs

For the ``corner situation'', we are lead to consider a solution in a domain of the type $\mathcal Q:=B_1'\times(-1,1)\times(0,1)$; the supersolution in this case is $v^+=|x'|^2+Mx_n^{1/2}$, with $M$ as in Paragraph \ref{barrier.paragraph}. The fact that the function is a supersolution follows plainly from \eqref{super.computation}. Assuming now that the boundary graph $\bar\theta$ over $B'_1$ takes values in $(-\delta,\delta)$ and 
\[
 \osc_{\partial_p\bar\Omega_T\cap\mathcal Q}\bar\psi\leq\delta,\quad\osc_{\mathcal Q}{\bar u}\leq 1,\quad \bar u(0,0)=\bar \psi(0,0)=0,
\]
we have $\bar u \leq v_\delta^+$ in $\partial_p\bar\Omega_T\cap \mathcal Q$,  since $\bar u=\bar\psi\leq \delta$ there; on the remaining part of the parabolic boundary of $\mathcal Q$ we use the fact that $v^+_\delta$ is larger than one, as in \eqref{perte}. Again, now the proof is similar as above.

In both cases, a scaling and iteration procedure like the one used in Paragraph \ref{quantify} allows us to prove the reduction of oscillation in a sequence of nested cylinders of the type $(\Omega\cap B_{r_j}(x_0))\times(0,\omega_j^2G(\omega_j/r_j))$, with $x_0\in\overline\Omega$. This leads to 
\begin{equation}\label{boundary.modulus.initial}
|u_\eps(x,t)-u_\eps(x_0,0)|\leq  \omega_u\big(\dist_{{\rm par},G}((x,t),(x_0,0))\big). 
\end{equation}
for every $(x,t)\in\Omega_T$. Moreover, opportune statements similar to above still hold in the case $\psi$ is H\"older continuous.

\vs

At this point we call $\widetilde R$ the smallest value of $R_0$ coming from the three different cases, ultimately a constant depending on $\data$ and $\partial\Omega$. Choose $r=\widetilde R/2$. Now $r$ is a constant depending only on $\data$ and $\partial\Omega$. By combining the boundary estimates \eqref{boundary.modulus} and \eqref{boundary.modulus.initial} with the interior estimate in Remark~\ref{rem.2} we obtain 
\begin{equation}\label{global.modulus}
 |u_\eps(x,t)-u_\eps(y,s)|\leq \omega_u\big(\dist_{{\rm par},G}((x,t),(y,s))\big)
\end{equation}
for every $(x,t),(y,s)\in\overline{\Omega_T}^p$, where $\omega_u$ depends on $\data,\epsilon,c_\ell,\omega_\psi,\|\psi\|_{L^\infty},\partial\Omega$. Indeed, if one of the points is in $\partial_p\Omega_T$, then \eqref{global.modulus} is either \eqref{boundary.modulus} or \eqref{boundary.modulus.initial}. In the case where both $(x,t),(y,s)\in\Omega_T$ we consider two different cases. Either the mutual distance of $(x,t)$ and $(y,s)$ is small compared to their distance to the boundary, in which case we use the interior estimate, or otherwise we can again use the boundary estimates.

Let us make this rigorous. Denote $z=(x,t),w=(y,s)$. If $\dist_{{\rm par},G}(z,w)\geq 1$, we are done by the maximum principle. Note now that if $\dist_{{\rm par},G}(z,w)\leq 1$, we have 
\[
 \dist_{{\rm par},G}(z,w)^{\max\{1,g_1/2\}}\leq\dist_{{\rm par}}(z,w)\leq \dist_{{\rm par},G}(z,w)^{\min\{1,g_0/2\}}.
\]
Observe that \eqref{improved.est.interior} can be written in terms of the parabolic $G$-distance as follows:
\[
 |u_\eps(z)-u_\eps(w)|\leq c\,\big[d^G_{z,w}\big]^{-\gamma\max\{1,g_1/2\}}\,\dist_{{\rm par},G}(z,w)^{\min\{1,g_0/2\}},
\]
where 
\[
 d^G_{z,w}:=\min\big\{\dist_{{\rm par},G}(z,\partial_p\Omega_T),\dist_{{\rm par},G}(w,\partial_p\Omega_T),1\big\}.
\]
If now $\dist_{{\rm par},G}(z,w)\leq \big[d^G_{z,w}\big]^{2\gamma\max\{1,g_1/2\}}$, that is, the mutual distance of $z$ and $w$ is small compared to their distance to the boundary, then we have 
\[
 |u_\eps(z)-u_\eps(w)|\leq c\,\dist_{{\rm par},G}(z,w)^{\min\{1,g_0-1\}/2}.
\]
On the other hand, when $\dist_{{\rm par},G}(z,w) > \big[d^G_{z,w}\big]^{2\gamma\max\{1,g_1/2\}}$ there exists a cylinder $Q_\rho^G(x_0,t_0)\ni z,w$ with $\rho=2\dist_{{\rm par},G}(z,w)^{\min\{1/2,1/g_1\}/\gamma}$ such that either $x_0\in\partial\Omega$ or the bottom of $Q_\rho^G(x_0,t_0)$ touches the initial boundary. Now using triangle inequality and the boundary estimates yields 
\[
 |u_\eps(z)-u_\eps(w)|\leq 2\,\omega_u(\rho) \leq 4\,\omega_u\big(\dist_{{\rm par},G}(z,w)^{\min\{1/2,1/g_1\}/\gamma}\big).
\]
Finally, we take the largest modulus of continuity $\omega_u$ for which all the conditions proved above are satisfied, and this proves \eqref{global.modulus}.  The proof in the H\"older case is similar, since in this case we can quantify all the moduli.

\section{Conclusion}\label{conclusion}

Call $u_j:=u_\eps$ for $\eps=1/j$, $j\in\N$, and similarly $\A_j, g_j, \phi_j$. From the results of the preceding section, that is, from the equi-boundedness of the sequence $\{u_j\}_{j\in\N}$ following from the maximum principle Corollary \ref{maximumprinciple} and the global equi-continuity coming from the results of Sections \ref{approximation} and \ref{boundary.continuity}, using Ascoli-Arzel\`a theorem we see that $u_j\to u$ uniformly in $C^0({\overline{\Omega_T}}^p)$ for some $u\in C^0({\overline{\Omega_T}}^p)$. Now all we have left to prove is that $u$ is a weak solution to \eqref{general.equation}$_1$, which follows easily from the next proposition. 

\begin{Proposition}
Let $u_j\in V^{2,G}_{\rm loc}(\Omega_T)\cap C^0({\overline{\Omega_T}}^p)$ be the solutions to \eqref{approx.eq} defined above. Suppose there exists a function $u$ such that $u_j\to u$ almost everywhere in $\Omega_T$. Then $Du_j\to Du$ almost everywhere.
\end{Proposition}
\begin{proof}
 Take $\mathcal K\Subset \Omega_T$ and choose a cutoff function $\varphi\in C_c^\infty(\Omega_T)$ such that $0\leq\varphi\leq 1$, $\varphi=1$ in $\mathcal K$, and $\|\partial_t\varphi\|_{L^\infty(\Omega_T)},\|D\varphi\|_{L^\infty(\Omega_T)}\leq c$ for some $c\geq 1$ depending on $\dist(\mathcal K,\partial_p\Omega_T)$. Let $j,k\in\N$ and test the weak formulations of $u_{j}$ and $u_{k}$ with $\eta=w_{j,k}\varphi$, where $w_{j,k}:=u_{j}-u_{k}$. This choice can be justified by standard methods such as Steklov averages. By subtracting we obtain 
 \begin{align*}
  0 &= -\int_{\Omega_T} w_{j,k}\partial_t(w_{j,k}\varphi)\dx\dt+ \int_{\Omega_T} \langle \A_{j}(Du_{j})-\A_{k}(Du_{k}), D(w_{j,k}\varphi)\rangle\dx\dt\\
  &= -\frac12\int_{\Omega_T}w_{j,k}^2\partial_t\varphi\dx\dt+\int_{\Omega_T} \big\langle\A_{j}(Du_{j})-\A_{k}(Du_{k}), D\varphi\rangle w_{j,k}\dx\dt\\
  &\quad\quad +\int_{\Omega_T} \big\langle\A_{j}(Du_{j})-\A(Du_{j}), Dw_{j,k}\big\rangle\varphi\dx\dt\\
  &\quad\quad\quad +\int_{\Omega_T} \big\langle\A(Du_{j})-\A(Du_{k}), Dw_{j,k}\big\rangle\varphi\dx\dt\\
  &\quad\quad\quad\quad +\int_{\Omega_T} \big\langle\A(Du_{k})-\A_{k}(Du_{k}), Dw_{j,k}\big\rangle\varphi\dx\dt=:I+II+III+IV+V.
 \end{align*}
Since $\|Du_{j}\|_{L^\infty({\Omega_T})} \leq c$ uniformly with respect to $j$ by Proposition~\ref{uniform.gradient.bound}, we also have 
\[
\|\A_{j}(Du_{j})\|_{L^\infty({\Omega_T})}\leq c g_{j}\left(\|Du_{j}\|_{L^\infty({\Omega_T})}\right)\leq c.
\]
Thus, by the definition of $\A_j$
\begin{multline*}
|I+II+III+V|\leq c\,\|u_{j}-u_{k}\|_{L^2({\Omega_T})} +c\|(\phi_{j}\ast\A)(Du_{j})-\A(Du_{j})\|_{L^2({\Omega_T})}\\[2mm]
+c\|\A(Du_{k})-(\phi_{k}\ast\A)(Du_{k})\|_{L^2({\Omega_T})}+c\,(1/j+1/k).
\end{multline*}
The first term on the right-hand side tends to zero as $j,k\to \infty$ by Lebesgue's dominated convergence theorem and the second and third by the properties of mollifiers; the last one is obvious. On the other hand, by \eqref{monotonicity} 
\[
IV\geq c\,\int_{\mathcal K}|V_g(Du_{j})-V_g(Du_{k})|^2\dx\dt.
\]
Thus
\[
c\,\int_{\mathcal K}|V_g(Du_{j})-V_g(Du_{k})|^2\dx\dt\leq IV\leq |I+II+III+V| \to 0
\]
as $j,k\to\infty$. We have shown that the sequence $\{V_g(Du_j)\}_{j\in\N}$ is Cauchy in $L^2({\mathcal K})$ and therefore there exists a function $w\in L^2({\mathcal K})$ such that $V_g(Du_j)\to w$ in $L^2({\mathcal K})$ as $j\to\infty$. This implies that there exists a (nonrelabeled) subsequence $V_g(Du_j)$ converging to $w$ almost everywhere in ${\mathcal K}$. Now the fact that $V_g$ has a continuous inverse yields 
\[
 Du_j=V_g^{-1}(V_g(Du_j))\to V_g^{-1}(w)=:v
\]
almost everywhere in ${\mathcal K}$.

Now, since $u_j\to u$ almost everywhere in ${\mathcal K}$, we have for any $\phi\in C_c^{\infty}({\mathcal K})$ that 
\[
\int_{{\mathcal K}}uD\phi\dx\dt=\lim_{j\to\infty}\int_{{\mathcal K}}u_jD\phi\dx\dt=-\lim_{j\to\infty}\int_{{\mathcal K}}Du_j\,\phi\dx\dt
=-\int_{{\mathcal K}}v\phi\dx\dt
\]
by Lebesgue's dominated convergence theorem and the definition of weak gradient, showing that $v=Du$. Thus, we have $Du_j\to Du$ almost everywhere in ${\mathcal K}$ for any ${\mathcal K}\Subset {\Omega_T}$, which implies that $Du_j\to Du$ almost everywhere in ${\Omega_T}$. 
\end{proof}

To conclude, \eqref{loc.lipschitz} follows from \eqref{Linfty1} simply using the local almost everywhere convergence of $Du_\eps$, and \eqref{exp.continuity} follows from \eqref{global.modulus} using the global uniform convergence of $u_\eps$.

\subsection{Weakening the assumptions}\label{weak.assumptions}
As mentioned in Remark \ref{more.general.conditions}, in this paragraph we show how to modify the proofs of the paper in order to obtain Theorems \ref{Ex} and \ref{Lip} for vector fields satisfying the weaker assumptions \eqref{assumptionsO.weak}.

We observe that assumptions \eqref{assumptionsO} are only used in order to have the analogous properties for the regularized vector field $\A_\eps$ defined in \eqref{regularization.A}. Moreover, \eqref{A.bound} and \eqref{A.ellipticity} trivially hold by taking $\xi_2=0$ in \eqref{assumptionsO.weak}. Thus, it suffices to show that under the assumptions \eqref{assumptionsO.weak} we still have \eqref{assumptionsO} for $\A_\eps$ with $g$ replaced by $g_\eps$ defined in \eqref{geps}.

We shall focus only on the convolution part of the vector field $\A_\eps$, since for the part involving the nondegenerate $\widetilde g_1$-Laplacian the corresponding estimates are classic and easy to verify. Therefore, we only need to prove \eqref{assumptionsO} with $\A$ replaced by $\phi_\eps\ast\A$ and $g(s)$ replaced by $\frac{g(s+\eps)}{s+\eps}s$. Using \eqref{assumptionsO.weak} we have 
\begin{align*}
 \langle D(\phi_\eps\ast\A)(\xi)\lambda,\lambda\rangle &= \lim_{h\to 0}\frac{1}{h}\int_{B_\eps(0)}\langle \A(\xi+h\lambda-\eta)-\A(\xi-\eta),\lambda\rangle\phi_\eps(\eta)\,d\eta\\
 &\geq 2^{g_0-2}\nu\int_{B_1(0)}\frac{g(|\xi-\eps\eta|)}{|\xi-\eps\eta|}\phi(\eta)\,d\eta \,|\lambda|^2
\end{align*}
and
\begin{align*}
 |D_j(\phi_\eps\ast\A)(\xi)| &= \lim_{h\to 0}\frac{1}{h}\left|\int_{B_\eps(0)}( \A(\xi+h\hat e_j-\eta)-\A(\xi-\eta))\phi_\eps(\eta)\,d\eta\right|\\
 &\leq 2^{g_1-2}L\int_{B_1(0)}\frac{g(|\xi-\eps\eta|)}{|\xi-\eps\eta|}\phi(\eta)\,d\eta.
\end{align*}
Hence, if we can show that 
\begin{equation}\label{approx}
 \int_{B_1(0)}\frac{g(|\xi-\eps\eta|)}{|\xi-\eps\eta|}\phi(\eta)\,d\eta \approx \frac{g(|\xi|+\eps)}{|\xi|+\eps}
\end{equation}
independently of $\eps$, we are done. 

Consider first the case $|\xi|\geq 2\eps$. This implies $|\xi-\eps\eta|\geq \frac{1}{3}(|\xi|+\eps)$,  and thus $|\xi-\eps\eta| \approx |\xi|+\eps$ so that \eqref{approx} holds. On the other hand, if $|\xi|<2\eps$, we have 
\begin{align*}
 \int_{B_1(0)}\frac{g(|\xi-\eps\eta|)}{|\xi-\eps\eta|}\phi(\eta)\,d\eta &\leq \sup_{B_1(0)}\phi\int_{B_3(\xi/\eps)}\frac{1}{|\xi/\eps-\eta|}\,d\eta \, \frac{g(|\xi|+\eps)}{\eps}\\
 &\leq c(n)\sup_{B_1(0)}\phi\,\frac{g(|\xi|+\eps)}{|\xi|+\eps}
\end{align*}
and 
\begin{align*}
 \int_{B_1(0)}\frac{g(|\xi-\eps\eta|)}{|\xi-\eps\eta|}\phi(\eta)\,d\eta &\geq \int_{B_{1/2}(0)\setminus B_{1/4}(\xi/\eps)}\frac{g(|\xi-\eps\eta|)}{|\xi-\eps\eta|}\phi(\eta)\,d\eta\\
 &\geq \inf_{B_{1/2}(0)}\phi \,|B_{1/2}(0)\setminus B_{1/4}(\xi/\eps)|\,\frac{g(\eps/4)}{|\xi|+\eps}\\
 &\geq c(n,g_1)\inf_{B_{1/2}(0)}\phi \,\frac{g(|\xi|+\eps)}{|\xi|+\eps}.
\end{align*}
Note that we can assume without loss of generality that $\sup_{B_1(0)}\phi\leq c$ and $\inf_{B_{1/2}(0)}\phi\geq 1/c$ for some $c\equiv c(n)>0$.

\vs

{\bf Acknowledgments}.
P.B. has been supported by the Academy of Finland and  the Gruppo Nazionale per l'Analisi Matematica, la Probabilit\`a e le loro Applicazioni (GNAMPA) of the Istituto Nazionale di Alta Matematica (INdAM). C.L. is supported by the Vilho, Yrj\"o and Kalle V\"ais\"al\"a Foundation. Part of the paper was conceived and written while the first author was visiting the Department of Mathematics and Systems Analysis of Aalto University, whose warm and friendly hospitality is gratefully acknowledged. He also acknowledges the support and the hospitality of the FIM at ETH Z\"urich in the Spring 2015, where part of the paper has been written. Finally, the authors would like to thank prof. P. Lindqvist for pointing out us some useful references, such as \cite{Bers, Dong}.

\end{document}